\documentclass[12pt]{amsart}

\usepackage{mathtools}
\mathtoolsset{showonlyrefs=true}
\usepackage[hmargin=0.8in,height=8.6in]{geometry}
\usepackage{amssymb,amsthm, times}
\usepackage{delarray,verbatim}
\usepackage{ifpdf}
\ifpdf
\usepackage[pdftex]{graphicx}
 \else
\usepackage[dvips]{graphicx}
 \fi

\usepackage{bm}

\usepackage{ifpdf}
\usepackage{color}
\definecolor{webgreen}{rgb}{0,.5,0}
\definecolor{webbrown}{rgb}{.8,0,0}
\definecolor{emphcolor}{rgb}{0.5,0.95,0.95}

\usepackage{hyperref}
\hypersetup{%
          colorlinks=true,
          linkcolor=webbrown,
          filecolor=webbrown,
          citecolor=webgreen,
          breaklinks=true}
\ifpdf \hypersetup{pdftex,
            pdfstartview=FitH, %%Fit, FitB, FitH
            bookmarksopen=true,
            bookmarksnumbered=true
} \else \hypersetup{dvips} \fi

\numberwithin{equation}{section}

\newtheorem{theorem}{Theorem}[section]
\newtheorem{proposition}{Proposition}[section]
\newtheorem{corollary}{Corollary}[section]
\newtheorem{remark}{Remark}[section]
\newtheorem{lemma}{Lemma}[section]
\newtheorem{assumption}{Assumption}[section]
\newtheorem{cond}{Condition}[section]

\numberwithin{remark}{section} \numberwithin{proposition}{section}
\numberwithin{corollary}{section}
\newcommand {\R}{\mathbb{R}}
\newcommand {\bR}{\mathbb{R}}

\newcommand {\cF}{\mathcal{F}}
\newcommand {\A}{\mathcal{A}}

\newcommand {\N}{\mathbb{N}}
\newcommand {\cL}{\mathcal{L}}
\newcommand {\cM}{\mathcal{M}}
\newcommand {\bN}{\mathbb{N}}

\newcommand {\p}{\mathbb{P}}
\newcommand {\bP}{\mathbb{P}}

\newcommand {\E}{\mathbb{E}}
\newcommand {\bE}{\mathbb{E}}

\newcommand{\diff}{{\rm d}}

\newcommand{\lev}{L\'{e}vy }

\begin{document}
\title{On stochastic control under {Poissonian intervention}: optimality of a barrier strategy in a general L\'evy model}

	\thanks{This version: \today.   }

\author[K. Noba]{Kei Noba$^*$}
\thanks{$*$\, Department of Fundamental Statistical Mathematics, The Institute of Statistical Mathematics, 
10-3 Midori-cho, Tachikawa, Tokyo 190-8562, Japan. Email: knoba@ism.ac.jp}
\author[K. Yamazaki]{Kazutoshi Yamazaki$^\dagger$}
\thanks{$\dagger$\, School of Mathematics and Physics, The University of Queensland, St Lucia, Brisbane, QLD 4072, Australia. Email: k.yamazaki@uq.edu.au}
%\date{}

\begin{abstract}  
We study a version of the stochastic control problem of minimizing the sum of running and controlling costs, where control opportunities are restricted to independent Poisson arrival times. Under a general setting driven by a general \lev process, we show the optimality of a periodic barrier strategy, which moves the process upward to the barrier whenever it is observed to be below it. The convergence of the optimal solutions to those in the continuous-observation case is also shown.
\ \ 
\\
\noindent \small{\noindent  {AMS 2020} Subject Classifications: 60G51, 93E20, 90B05 }
\end{abstract}

\maketitle

\section{Introduction} \label{section_introduction}

%{cite: Albrecher,  Bauerle and Thonhauser: Optimal dividend payout in random
%discrete time 2011.}

Stochastic control aims to obtain an optimal dynamic strategy in cases of uncertainty. In its typical formulation, the problem reduces to obtaining an adapted control process that maximizes/minimizes the expected total reward/cost, which depends on the paths of the controlling and controlled processes. The continuous-time stochastic control research, active in various fields such as financial/actuarial mathematics and research on inventory models, has been developed along with stochastic analysis and differential equations theory. In contrast with its discrete-time counterpart, for which numerical approaches are typically required, various analytical approaches, such as It\^o calculus and first passage analysis, are available in continuous-time models to obtain explicit results.

%Stochastic control aims to obtain an optimal dynamic strategy in the existence of uncertainty. In its typical formulation, the problem reduces to obtaining an adapted control process that maximizes/minimizes the expected total reward/cost, which depends on the paths of the control and controlled processes.  The continuous-time stochastic control research, active in various fields such as financial/actuarial mathematics and inventory models, has been developed along with stochastic analysis and differential equations theory. Contrary to the discrete-time counterpart, for which numerical approaches are typically required, various analytical approaches  such as It\^o calculus and first passage analysis are available in continuous-time models to obtain explicit results.

{Poissonian observation/intervention models have been developed to explore the interface between continuous-time and discrete-time models. The earliest papers on this model include Wang \cite{H_Wang} and Dupuis and Wang \cite{Dupuis_Wang} for Brownian motion models. More recently, these results have been extended to spectrally one-sided L\'evy models, as discussed in works by, among others, \cite{ABT, Albrecher,Avanzi_Cheung_Wong_Woo,   Lkabous, PPSY,PPY, Perez_yamazaki_AMO, Zhang_Han, ZCY, ZWYC}. For a comprehensive survey on this subject, see Saarinen \cite{Saarinen} and the references therein.}
%In Poissonian observation models, 
In the Poissonian model, instead of allowing the decision maker to observe the state process continuously and control it at all times, these opportunities are given only at independent Poisson arrival times. 
%\blue{With the selection of exponential interarrival times, the problem often continues to be analytically tractable, thanks mainly to the memoryless property of the exponential random variables.} 
Although this assumption of Poisson arrivals is indeed restrictive in real applications, it provides a more flexible approach for approximating the discrete-time counterpart (with deterministic interarrivals) than the classical continuous-time model. As confirmed numerically in studies such as \cite{LYZ2}, approximation via Poisson arrivals (as a special case of {Erlangization} \cite{Carr_random,Dong, KP}) often achieves accurate approximation of the discrete-time model in stochastic control problems.  % We refer the reader to \cite{Czarna, YSL2, YSL} concerning the delays  in the so-called Parisian ruin model with restricted observation times, closely related to the Poissonian observation models.
 %, closely related to the random observation models.

%Poissonian observation models have been developed to explore the interface between continuous-time and discrete-time models (see TODO). 
%Instead of allowing the decision maker to observe the state process continuously and control it at all times, these opportunities are given only at independent Poisson arrival times. With the selection of exponential interarrival times, the problem often continues to be analytically tractable, mainly thanks to the  memoryless property of the exponential random variables. Although this assumption of Poisson arrivals is indeed restrictive in real applications,  it provides a more flexible approach to approximate the discrete-time counterpart (with deterministic interarrivals) than the classical continuous-time model.  As confirmed numerically in the literature such as TODO, approximation via Poisson arrivals (as a special case of Erlanization TODO) often achieves accurate approximation to the discrete-time model in stochastic control problems. 

This paper studies the classical stochastic control problem, described as follows. 
Given a stochastic process $X = \{ X_t: t \geq 0 \}$, the objective is to choose a strategy $\pi = \{R_t^\pi : t \geq 0\}$ to  minimize the total expected values of the running cost $\int_0^\infty e^{-qt} f(U_t^\pi) \diff t$ and the controlling cost  $\int_{[0, \infty)} e^{-qt} \diff R_t^\pi$, where $U^\pi := X+R^\pi$ is the controlled process when  $\pi$ is applied.  {More precisely, we want to minimize over $\pi$ the expected sum
$v_\pi (x):=\E_x \left[ \int_0^\infty e^{-qt}f(U^\pi_t)\diff t + C\int_{0}^\infty e^{-qt} \diff R^\pi_t \right]$ for $C \in \mathbb{R}$.}
{This framework enables the modeling of various optimization scenarios by suitably selecting the process $X$.}
%{Various problems can be modeled in this framework by appropriately choosing the process $X$. TODO} 
See \cite{Bensoussan_Tapiero,Benkerouf_Bensossan,Bensoussan_Liu_Sethi} for inventory models and \cite{Cai, Jeanblanc, Mundaca} for financial applications.

This problem has been studied in several papers when $X$ is a spectrally negative \lev process (i.e., a \lev process with only negative jumps). Under the assumption that the running cost function is convex, the barrier strategy, with the lower barrier $b^*$ selected to be a unique root of
\begin{align}
\bE_b \left[\int_0^\infty e^{-qt} f^\prime_+ (U^b_t) \diff t \right] + C = 0, \label{opt_b_review}
\end{align}
is  optimal. Here, $U^b$ is the reflected process starting at $b$.
 %with a lower barrier $b$, which is also called a draw-up or an excursion in the literature. 
 {Interestingly, this optimality result continues to hold in different formulations with additional constraints on the admissible strategies.}
%Interestingly, this 
%optimality result continuous to hold under non-standard formulations. 
In a version where $R^\pi$ is restricted to being absolutely continuous with respect to the Lebesgue measure with a given density bound \cite{Hernandez}, the same optimality result holds, with $U^b$ being the so-called refracted processes \cite{Kyp_loeffen,NobPerezYamazaki, Perez_yamazaki_yu,Wang, Wang_Wang_Chen, YSL}. The Poissonian observation version we consider in this paper has been solved by \cite{Perez_Yamazaki_Bensoussan} for the spectrally negative case. In this case, $U^b$ is a version of the reflected process that is pushed to $b$ whenever it is observed to be below it. By selecting the barrier using \eqref{opt_b_review}, this version of the barrier strategy, which we call the periodic barrier strategy,  has been shown to be optimal.

The results described above all rely on the so-called scale function (see \cite{Ber1996,Egami_Yamazaki, Kyp2014, Yam2015}), which makes sense only for spectrally one-sided \lev processes. 
%The expected cost under barrier strategies can be written concisely via the scale function, with which the smoothness of the value function and its optimality can be verified in a straightforward way. 
However, the spectrally negative assumption is often unrealistic in real applications. For example, financial asset prices are empirically known to have both positive and negative jumps (see \cite{CGMY}); also, water storage levels of dams experience both positive and negative jumps, due to rainfall and surges in consumption. See also the introduction of \cite{CLW} for the application of processes of two-sided jumps in modeling the surplus of an insurance company. % actuarial science.

%The results described above all rely on the so-called scale function (see TODO), which makes sense only for spectrally one-sided \lev processes. The expected cost under barrier strategies can be written concisely via the scale function, with which the smoothness of the value function and its optimality can be verified in a straightforward way.
%However, the spectrally negative assumption is often unrealistic in real applications. For example, financial asset prices are empirically known to have both positive and negative jumps (see TODO). Water storage levels of dams experience both positive and negative jumps due to rainfalls and surge of consumption.

Although the existing results for a general \lev process in stochastic control are significantly limited in comparison with diffusion and spectrally one-sided \lev models, the problem described above has recently been solved for a general \lev process in the continuous-observation setting. Noba and Yamazaki \cite{NobYam2020} have shown that the classical barrier strategy described in  \eqref{opt_b_review} continues to be optimal even in the presence of positive jumps. It is thus a natural conjecture that the form of optimal strategy is invariant to the existence of upward jumps.
% and the same optimality result holds more generally, even under non-standard formulations.
%Although the existing results for a general \lev process in stochastic control are significantly limited in comparison to diffusion and spectrally one-sided \lev models, the problem described above has recently been solved for a general \lev process in the continuous-observation setting. It has been shown, in  \cite{NobYam2020},  that the classical barrier strategy described with \eqref{opt_b_review} continues to be optimal even with the presence of positive jumps.   It is thus a natural conjecture that the form of optimal strategy is invariant to the existence of upward jumps and the same optimality result holds more generally, even under non-standard formulations.
The objective of this paper is to verify this conjecture. We solve the Poissonian observation case for a general \lev process $X$ with both positive and negative jumps, generalizing the results of \cite{Perez_Yamazaki_Bensoussan} and \cite{NobYam2020} simultaneously and provides a unified way of expressing the optimal strategy.  Despite the obvious difficulty over the continuous observation model for which many analytical results are available for classical reflected processes, we provide a more concise proof than those given in \cite{NobYam2020}.

The remainder of the paper is organized as follows. In Section \ref{Sec02}, we formally define the problem under consideration. In Section \ref{section_barrier}, we define periodic barrier strategies and obtain their key properties. Then, in Section \ref{section_optimality}, we select the barrier and demonstrate its optimality. In Section \ref{section_convergence}, we show the convergence to the results in the classical setting as the rate of observation approaches infinity. These results are confirmed with numerical experiments in Section \ref{section_numerics}. {Finally, we conclude the paper in Section \ref{section_concluding_remarks}.} {Some proofs are deferred to the appendix. Throughout the paper, we let $g^\prime_+ (\cdot)$ and $g^\prime_- (\cdot)$ be the right-hand and left-hand derivatives of any function $g$ whenever they make sense.}

%The rest of the paper is organized as follows. In Section \ref{Sec02}, we formally define the problem under consideration. In Section \ref{section_barrier}, we define periodic barrier strategies and obtain their key properties. Then, in Section \ref{section_optimality}, we select the barrier and demonstrate the optimality. In Section \ref{section_convergence}, we show the convergence as the rate of observation goes to infinity to the ones in the classical setting. These results are confirmed with numerical results in Section \ref{section_numerics}. We conclude the paper in Section \ref{section_concluding_remarks}.

\section{Problem}\label{Sec02}

%\subsection{Problem}\label{Sec201}
%\blue{[Just delete this paragraph?]} In this paper, we consider an inventory model where the cumulative demand $D = \{ D_t^\pi: t \geq 0 \}$ follows a general \lev process. \blue{[let's discuss alternative word for inventory]}
%With its initial value $x \in \R$, the inventory (in the absence of control) follows a {(one-dimensional)} \lev process
%\[
%X_t := x - D_t, \quad t \geq 0.
%\]

%\blue{[rephrased the paragraph.]}
%\blue{[changed some to differentiate from the MOR paper]}
Let  $X=\{X_t : t \geq 0 \}$  be a (one-dimensional) \lev process defined on a probability space $(\Omega , \mathcal{F}, \p)$. 
For $x\in\R$, we denote by $\p_x$ the law of $X$ when its initial value is $x$, and write $\mathbb{P} = \mathbb{P}_0$ for the case $x = 0$.  Let $\Psi$ be the \emph{characteristic exponent} of $X$; i.e.\
$e^{-t\Psi(\lambda)}=
\E \left[e^{i\lambda X_t}\right]$, $\lambda \in \R$ and $t\geq 0$.
It is known to admit the form 
\begin{align}
\Psi (\lambda) := -i\gamma\lambda +\frac{1}{2}\sigma^2 \lambda^2 
+\int_{\R \backslash \{0\} } (1-e^{i\lambda z}+i\lambda z1_{\{|z|<1\}}) \Pi(\diff z) , \qquad \lambda\in\R, \label{202a}
\end{align}
for some $\gamma\in\R$, $\sigma\geq 0$, and a \lev measure $\Pi $ on $\R \backslash \{0\}$ satisfying
$\int_{\R\backslash \{ 0\}}(1\land z^2) \Pi (\diff z) < \infty$. 
%\blue{[maybe ok to delete the following sentence.]} Recall that the process $X$ has bounded variation paths if and only if $\sigma= 0$ and $\int_{|z|<1} |z|\Pi(\diff z)<\infty$. When this holds, we can write
%\begin{align}
%\Psi(\lambda) := -i\delta \lambda +\int_{\R \backslash \{0\} }  (1-e^{i\lambda z})\Pi (\diff z),  \label{204}
%\end{align}
%where $\delta := \gamma-\int_{{(-1,1) \backslash \{0\}}}z \Pi(\diff z)$.

We consider a {version of the} stochastic control problem 
%with Poissonian controlling opportunities 
defined as follows.
%} control problem with proportional replenishment costs (without fixed costs).  
%We assume that the inventory is replenished at
{The set of control opportunities}
\[\mathcal{T}_{\eta}:= \{T(k): k\in\N \} \] {are} given by the arrival times  of a Poisson process $N^{\eta}=\{N^{\eta}_t : t\geq 0\}$ with intensity $\eta>0$ which is independent of $X$. 
{In other words, {the interarrival times} $\{T(k)-T(k-1): k \in\bN \}$ are an i.i.d.\ sequence of exponential random variables with intensity $\eta$ where {we let} $T(0)=0$ {for notational convenience}.} 
Let $\mathbb{F} := \{ {\mathcal{F}_t}: t \geq 0 \}$ be the natural filtration generated by $(X, N^{\eta})$. %\blue{[change $r$ to $\lambda$? $r$ is often used for discount in finance.]}
%{[Since $\lambda$ is used as a variable, I have replaced it with $\eta$. ]}
A \emph{strategy}, representing the cumulative amount of {controlling}, $\pi=\{ R^\pi_t: t\geq 0\}$ is a
% nondecreasing, right-continuous, and $\mathbb{F}$-adapted process starting at $R^{\pi}_{0-}=0$ \blue{[ok to delete right-continuous and  $R^{\pi}_{0-}=0$ because they are clear from the following form?]}{[Non-decrease and $\mathbb{F}$-adapted are also almost obvious, so may as well be deleted]} which is of the form
 {process of the form}
 \begin{align}
 R^\pi_t = {\int_{[0,t]} \nu^\pi_s \diff N^{\eta}_s } {= \sum_{0 \leq s \leq t: N^\eta_s \neq N^\eta_{s-} } \nu_s^\pi}\qquad t\geq 0,  \label{A001}
 \end{align}
 for some c\`agl\`ad {(left-continuos with right limits)} and non-negative $\mathbb{F}$-adapted process $\nu^\pi=\{\nu^\pi_t : t\geq 0\}$, {where it is understood that $N^\eta_{0-}=0$.} 
{The corresponding controlled process becomes}
\begin{align}
U^\pi_t=X_t +R^\pi_t,\quad t\geq 0. 
\end{align}
{We focus on the case where one can control the state process in one direction, and hence $\nu_s^\pi \geq 0$ a.s.\ for all $\pi \in \mathcal{A}$, which is standard as in \cite{Bensoussan_Tapiero, Benkerouf_Bensossan, Bensoussan_Liu_Sethi, Hernandez}. Such an assumption is applicable in many inventory models where only replenishment is allowed, as well as in dam management scenarios where the water level can only be decreased by the decision maker.}

\par
%\blue{[changed sentences to differentiate from MOR paper]} 
For a given discount factor $q > 0$ and initial value $x \in \R$, the objective is to minimize %the net present value of total costs
\begin{align*}
v_\pi (x):=&\E_x \left[ \int_0^\infty e^{-qt}f(U^\pi_t)\diff t + C\int_{0}^\infty e^{-qt} \diff R^\pi_t \right] \\
=&{\E_x \left[ \int_0^\infty e^{-qt}f(U^\pi_t)\diff t + C\sum_{0 \leq t < \infty: N^\eta_t \neq N^\eta_{t-} } e^{-qt } \nu_t^\pi \right],}
\end{align*}
%\blue{[$\int_{[0, \infty)}$ can be changed to $\int_0^\infty$ because $R$ does not jump at time zero?]}
%{[I think it is fine.]}
which is the sum of running cost for a given measurable function $f : \R \to\R$ and a controlling cost/reward for a unit cost/reward $C \in \R$ {(cost if {it is} positive and reward if negative)}.  Let $\A$ be the set of all admissible strategies satisfying the constraints described above as well as {the integrability condition:}
\begin{align}
\E_x &\left[ \int_0^\infty e^{-qt}{|f(U^\pi_t)|}\diff t + \int_0^\infty e^{-qt} \diff R^\pi_t \right] \\
&\qquad \qquad {=
\E_x \left[ \int_0^\infty e^{-qt}{|f(U^\pi_t)|}\diff t + \sum_{0 \leq t < \infty: N^\eta_t \neq N^\eta_{t-} } e^{-qt } \nu_t^\pi \right]}
< \infty. \label{1}
\end{align}
%
%We fix a discount factor $q > 0$ and a unit cost/reward of controlling $C \in \R$ {(cost if {it is} positive and reward if negative)}. Associated with each {strategy} $\pi \in \A$, the cost of inventory is modeled by $\int_0^\infty e^{-qt} f(U^\pi_t)\diff t $ for a measurable {inventory cost} function $f : \R \to\R$
% and that of controlling is given by $C\int_{[0,\infty)} e^{-qt}\diff R^\pi_t$. The problem is to minimize their expected sum
%\begin{align*}
%v_\pi (x):=\E_x \left[ \int_0^\infty e^{-qt}f(U^\pi_t)\diff t + C\int_{[0, \infty)}e^{-qt} \diff R^\pi_t \right], \quad x\in\R, 
%\end{align*}
%over the set of all admissible {strategies} $\A$ that satisfy all the constraints described above and {the integrability condition:}
%\begin{align}
%\E_x \left[ \int_0^\infty e^{-qt}{|f(U^\pi_t)|}\diff t + \int_{[0, \infty)}e^{-qt} \diff R^\pi_t \right]< \infty,\quad x\in\R. \label{1}
%\end{align}
{The aim of the problem} is to obtain the {\emph{(optimal) value function} }
\begin{align}
v(x):=\inf_{\pi\in\A}v_\pi(x),\quad x\in\R, 
\end{align}
and {an} optimal {strategy} $\pi^\ast$ {such that $v_{\pi^*}(x) = v(x)$ (if such a {strategy} exists)}.

%\subsection{Standing assumptions}

For the {running cost} function $f$, the unit cost/reward $C$ and the \lev process $X$, we impose the same conditions as {those assumed in }\cite{NobYam2020}; {similar conditions are} commonly assumed in the literature (see \cite{Benkerouf_Bensossan, Bensoussan_Liu_Sethi, Hernandez, Yam2017}).

%Throughout the paper, we assume the following on the function $f$ {and the unit cost/reward $C$}.  Note that this is commonly assumed in the literature (see, e.g., \cite{Benkerouf_Bensossan, Bensoussan_Liu_Sethi, Hernandez, Yam2017}).

\begin{assumption}[Assumption on $f$ and $C$] \label{Ass201}
%We assume the following throughout the paper:
\begin{enumerate}
\item[(1)] The function $f$ is convex.
\item[(2)] %The function $f$ has at most polynomial growth in the tail.  That is to say, t
There exist $k_1$, $k_2 {> 0}$ and $N\in\N$ such that $|f(x)|\leq k_1+k_2 {|x|^N}$ for all $x \in\R$. 
\item[(3)]  We have $f^\prime_+ (-\infty)<-Cq< f^\prime_+(\infty)$ where $f^\prime_+(-\infty):=\lim_{x\to-\infty} f^\prime_+ (x)\in[-\infty,\infty)$ and  $f^\prime_+(\infty):=\lim_{x\to\infty}f^\prime_+(x)\in(-\infty,\infty]$. % {which exist by (1).} 
\end{enumerate}
\end{assumption}

{
\begin{remark} 
Examples of $f$ satisfying the above assumptions include classical examples such as $f(x) = x^2$ and $f(x) = |x|$, as well as asymmetric functions used for our numerical examples \eqref{examples_f} in Section \ref{section_numerics}.
%Examples of $f$ satisfying the above assumptions include classical examples such as $f(x) = x^2$ and $f(x) = |x|$, as well as asymmetric functions as used for our numerical examples in Section \ref{section_numerics}.
\end{remark}
}
Note that the right- and left-hand derivatives $f^\prime_+ (x)$ and $f^\prime_- (x)$, respectively, for all $x \in \R$ as well as their limits  are well-defined by Assumption \ref{Ass201}(1). {Assumption \ref{Ass201}(3) is necessary to avoid the optimality of a trivial strategy {and the case optimal strategy does not exist}; see \cite[Remark 1]{NobYam2020}.} {More precisely, when this assumption is violated, the optimal strategy is never to modify the process or to move the process to an arbitrarily large value.}

\begin{assumption}[Assumption on $X$]  \label{Ass202}
%We assume the following regarding the \lev process $X$.
\begin{enumerate}
\item[(1)] 
%The process
 $X$ is not a {(driftless)} compound Poisson process. 
%\blue{[later we can relax this by taking limits?]}{[I added a remark.]}
\item[(2)] For some $\bar{\theta}>0$,
$\int_{\R\backslash(-1,1)}e^{\bar{\theta} |z|}\Pi(\diff z) < \infty$. 
%\label{tail_assump}
\end{enumerate}
\end{assumption}

\begin{remark} \label{remark_mass} %\blue{[added this remark]}
By Assumption \ref{Ass202}(1) and \cite[Proposition I.15]{Ber1996}, the potential measure of $X$ has no atoms. This also shows $\bP_x (X_{T(1)}= b)
={\eta}\bE_x \left[ \int_0^\infty e^{-{\eta}t} 1_{\{X_t =b\}} \diff t\right]=0$ for all $x,b \in \R$.
\end{remark}

Assumption \ref{Ass202}(2) together with \cite[Theorem 3.6]{Kyp2014} guarantees {the finiteness of} $\E [ \exp (\bar{\theta} {|X_1|})]$ and {also that of $\E \left[|X_1|\right]$ (since $\exp(x) \geq x$ for $x\geq0$).}
%, we also have $\E \left[|X_1|\right]<\infty$.

\begin{remark}\label{Rem201} %\blue{[rephrased]}
From Assumptions \ref{Ass201}(2) and \ref{Ass202}(2) and by the proof of \cite[Lemma 11]{Yam2017}, {the expectation}
$\E_x \left[\int_0^\infty e^{-qt} |f(X_t)|\diff t \right]$ is finite %for all $x \in\bR$ and the map
%$
%x\mapsto \bE_x \left[\int_0^\infty e^{-qt} |f(X_t)|\diff t \right]
%$
and it is at most of polynomial growth as $x\uparrow \infty$ and $x\downarrow-\infty$. 
\end{remark}

%\blue{[deleted the remark about assumption of compound Poisson]}
%\begin{remark}
%\blue{[under construction] We assume Assumption \ref{Ass202} (1) for the brevity of the paper. The main result (Theorem \ref{Thm501}) remains to hold even for the case $X$ is a driftless compound Poisson process. This can be shown by approximating using a convergent sequence of drifted compound Poisson processes for which the main result holds by following the arguments given in Section {5} of \cite{NobYam2020}.}
%{In the classical barrier case, the proof of the approximation was divided into cases where $X$ is a negative subordinator and other cases, but this time there is no need for such a case separation and the argument of the approximation is simpler. }
%
%
%%The main result, which we will write later, holds even without Assumption \ref{Ass202} (1). We will discuss this in {Section \ref{Sec502}}. %\blue{[We will make B later?]}{[I have a mistake. I changed from B to Section \ref{Sec502}.]}
%\end{remark}
%{As addressed in the introduction, our key tools are the expressions of the derivatives of the expected costs under the barrier strategy in terms of this random variable. }

\section{Periodic barrier strategies} \label{section_barrier}

%\blue{[rephrased the whole section.]}

Our objective is to show the optimality of a periodic barrier strategy $\pi^b$ for a suitable selection of the barrier {$b \in \R$} in the considered stochastic control problem.

\begin{figure}[htbp]
\begin{center}
\begin{minipage}{1.0\textwidth}
\centering
\begin{tabular}{c}
 \includegraphics[scale=0.7]{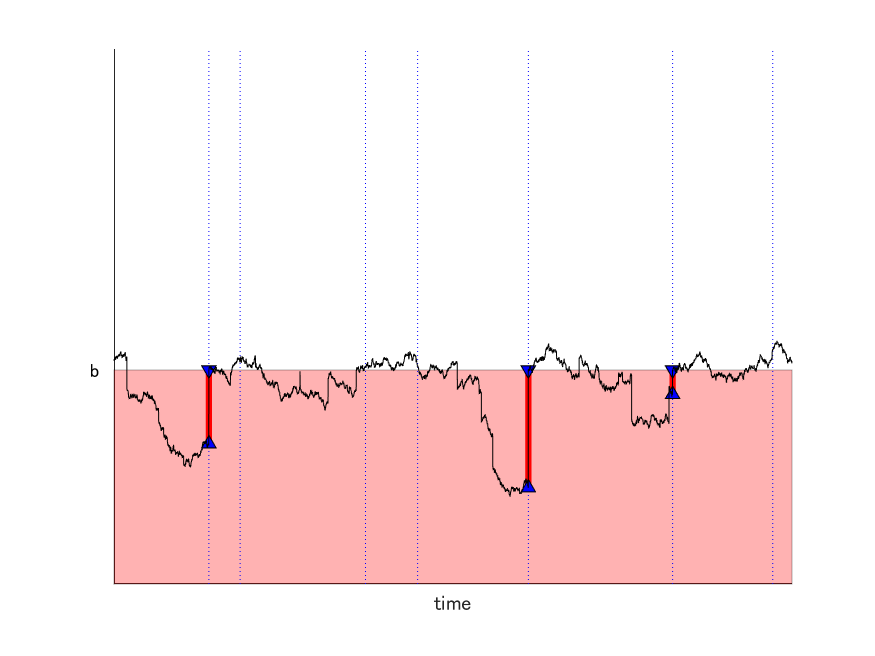}
\end{tabular}
\vspace{-4mm}
\caption{Sample path of $U^b$.
%%The solid black trajectory shows the path of $X$ and the piecewise horizontal blue lines show player $P$'s most recent %information on $X$
%%{exercise opportunity, whose} 
Control opportunities $\mathcal{T}_{\eta}$ are shown by dotted vertical lines.  The control times $\{T_b^{(n)}: n\in\bN \}$ and control sizes $\Delta R^b$ are indicated by the vertical red lines.} 
\label{plot_illustration}
\end{minipage}
\end{center}
\end{figure}

Fix $b \in \R$. A periodic barrier strategy $\pi^b$ pushes upward the process to $b$ whenever it is observed to be below $b$ {(see Figure \ref{plot_illustration})}. The epochs of controlling $\{T_b^{(n)}: n\in\bN \} \subset{\mathcal{T}_{\eta}}$  are given by %More formally, it is uniquely defined by the controlling times given as 
%An intuitive definition of the strategy is via a recursive algorithm.
%%We consider the periodic barrier strategy which replenishes any shortage below $b$ at the Poissonian opportunities $\mathcal{T}_{{\eta}}$, which can be recursively defined as follows.
%The strategy with the barrier $b \in \R$ pushes the process up to $b$ at 
a sequence of  $\mathbb{F}$-stopping times, recursively defined as follows: {with $T_b^{(0)} :=0$,
%\begin{align}
%\begin{aligned}
%&T_b^{(0)}={0}, \quad T_b^{(1)} = \inf\{t > 0 : {X_t} <b, N^{{\eta}}_t \neq N^{{\eta}}_{t-}\},\\
%T_b^{(n)} =& \inf\{t >T_b^{(n-1)}:b+(X_t-X_{T_b^{(n-1)}})<b, N^{{\eta}}_t \neq N^{{\eta}}_{t-}\},\quad n\geq 2.
%\end{aligned}
%\label{32}
%\end{align}
%{[I think it is easier to see if we write it in terms of $\mathcal{T}_{{\eta}}$. How about
\begin{align}
\begin{aligned}
%&T_b^{(0)}={0}, \quad 
%T_b^{(1)} &= \inf\{t \in \mathcal{T}_{\eta}: {X_t} <b \},\\
T_b^{(n)} &= \inf\{t \in \mathcal{T}_{\eta}: t >T_b^{(n-1)}, 
\tilde{X}_t^{(n-1),b}
%b+(X_t-X_{T_b^{(n-1)}})
<b\},\quad n\geq 1,
\end{aligned}
\label{32}
\end{align}
where 
\[
\tilde{X}_t^{(m),b} := 
\begin{cases}
X_{t},       & m = 0,\\
b+(X_t-X_{T_b^{(m)}}), & m \geq 1,
\end{cases}
\] 
is a parallel shift of $X$ so that it starts from $b$ at $T_b^{(m)}$ when $m \geq 1$. The strategy $\pi^b$ modifies $X$ by adding at $T_b^{(n)}$ the shortage $b - \tilde{X}_{T_b^{(n)}}^{(n-1),b}$ so that the path of the controlled process is the concatenation of $(\tilde{X}^{(m),b})_{m \geq 1}$. The corresponding control and controlled processes can be written, respectively,
\begin{align}
R_t^b &:= R_t^{\pi^b} = \sum_{n = 1}^\infty (b - \tilde{X}_{T_b^{(n)}}^{(n-1),b}) 1_{\{ T_b^{(n)} \leq t \}}, \\
U^b_t &:= U^{\pi^b}_t = X_t + R^b_t
= \sum_{n=1}^\infty \tilde{X}_t^{(n-1),b} 1_{\{  t\in [T^{(n-1)}_b, T^{(n)}_b )\}}.
%\begin{cases}
%X_{t} , \qquad         &\qquad t \in [ 0 , T^{(1)}_b ),\\
%\tilde{X}_t^{(n-1),b} ,\qquad   & t\in [T^{(n-1)}_b, T^{(n)}_b )\text{ with } n\geq 2 .
%\end{cases}
\end{align}
%becomes {what we call} the L\'evy process with %a lower barrier $b$. 
%Parisian reflection at $b$. %\blue{lets think later about how to call this.  }
{Note that $X$ and $N^\eta$ do not jump at the same time.}

%]}

Alternatively, in terms of the c\`agl\`ad {$\mathbb{F}$-adapted process} $\nu^b=\{\nu^b_t : t\geq 0\}$ with 
\begin{align}
\nu^b_t =
\begin{cases}
\left(b - X_{t-}\right)^+, \qquad         &\qquad t \in [ 0 , T^{(1)}_b ],\\
{\left( X_{T_b^{(n-1)}}- X_{t-}\right)^+, } \qquad   & t\in(T^{(n-1)}_b, T^{(n)}_b ]\text{ with } n\geq 2 ,
\end{cases}
\end{align}
where $x^+ := x \vee 0$ and it is understood that $X_{0-}=X_0$, 
%\blue{[do we need this?]}
%{[I have included this sentence to prove that periodic barrier strategies satisfies \eqref{A001}. We might as well pull this out as the paper might be better summarised in a shorter form.]}
%the strategy $\pi^b$ can be written as
it can be also written
\begin{align}
R^b_t =  {\int_{[0,t]}} \nu^b_s \diff N^{\eta}_s, \qquad t\geq 0.  \label{R_b_nu}
\end{align}
%{[It would be better to write $\int_{[0, t]}$ to emphasise that $t$ is included in the integral interval. ]}
%\blue{Clearly, $\nu^b$ is c\`agl\`ad and $\F$-adapted, and hence it is a feasible }.
%For simplicity, we write $R^b$ for $R^{\pi^b}$. 

{
\begin{remark} \label{remark_simple_expression_R}
{
In terms of the minimum of $X$ observed until time $t$, we can also write
\[
R_t^b = \max_{1 \leq k \leq N_t^\eta} (b-X_{T(k)})^+, \quad t \geq 0,
\]
%{[Would it be simpler and better to use this as the definition?]}
and  $T^{(n)}_b$ as the $n$-th jump time of $R^b$. This expression will be used to show the convergence to the classical case in Section \ref{section_convergence}.}
\end{remark}
}

%\blue{[moved here]} 
For the rest of the paper, {we denote the expected}
%we let the expected 
total cost under the periodic barrier strategy $\pi^b$ {by} 
%be written
\begin{align}
v_b (x):= v_{\pi^b} (x), \quad x \in \R.
\end{align} 

%As is commonly pursued in classical singular control problems,  our objective is to show the optimality of a barrier strategy
%$\pi^b$ for some $b\in\R$. 

%{It is is given in Appendix \ref{SecA01}.  }

{
%We conclude this section with the following three lemmas whose proofs are given in Appendix \ref{appendix_proof}. In particular, 

%\blue{[I made separate lemmas for integrability and admissibility. I also moved the lemma for the integrability of the derivative here.]}

{We now show the admissibility of periodic barrier strategies along with related results.  The proof of the following lemma is deferred to Appendix \ref{SecA01}.
\begin{lemma}\label{Lem301}
For $x,b \in \R$, (i) $\E_x\left[ \int_0^\infty e^{-qt} |f(U^b_t)|\diff t   \right] < \infty$ and (ii) $\E_x\left[ \int_0^\infty e^{-qt}  \diff R^b_t  \right] < \infty$, {(iii) ${x \mapsto v_{b}(x)}$ is at most of polynomial growth.}
%For $b\in\bR$, the strategy $\pi^b$ is admissible. 
\end{lemma}

%\blue{[added the following.]}
%\blue{Regarding the growth condition of $v_{b^*}$, we have the following. The proof is deferred to Appendix \ref{proof_lemma_poly_growth}.
%\begin{lemma} \label{lemma_poly_growth} The function 
%\end{lemma}
%}

As a corollary of the above, we also have the following. Thanks to Assumption \ref{Ass201}(1), this can be shown exactly in  the same way as the proof of \cite[Lemma 4]{NobYam2020} and thus we omit the proof.
% Thus, we omit the detail of the proof.
\begin{corollary}\label{Lem303} %\blue{[moved here]}
%[I changed the statement.]
For $x,b \in \R$,  we have 
$\bE_x \left[\int_0^\infty e^{-qt} \left|f^\prime_+ (U^b_t)\right|\diff t \right]<\infty$.
\end{corollary}
Lemma \ref{Lem301} together {with} \eqref{R_b_nu} shows the following.
\begin{proposition}\label{Lem301_2}
For $b\in\bR$, the strategy $\pi^b$ is admissible. 
\end{proposition}

%The admissibility of $\pi^b$ is verified in Lemma \ref{Lem301}.
}

%\blue{[moved the following results that hold for arbitrary $b$]}
Let $T_b := T_b^{(1)}$ be the first control time under the policy $\pi^b$. 
%\blue{We show the following lemma to obtain the derivative of $v_b$; see Appendix \ref{proof_Lem302} for the proof.}
{We conclude this section with the expression of the slope of $v_b$ written in terms of $T_b$ and the {uncontrolled} \lev process $X$.}

\begin{proposition}\label{Lem404} %\blue{[changed from Lemma to Proposition]}
{For $b \in \R$, the function $v_b$ is continuously differentiable with its derivative}
%For $x, b\in \bR$, we have 
\begin{align}
v_b^\prime (x)%\lim_{\varepsilon \to 0}\frac{v_b(x+\varepsilon)-v_b(x)}{\varepsilon}
= \bE_x \left[\int_0^{T_b} e^{-qt} f^\prime_+(X_t) \diff t \right]-C\bE_x \left[ e^{-qT_b}\right], \quad x \in \R.  \label{20}
\end{align}
%which belongs to $C(\bR)$. 
\end{proposition}

{The proof of Proposition \ref{Lem404} requires the following continuity result of $T_b$; its proof is deferred to Appendix \ref{proof_Lem302}.}
\begin{lemma}\label{Lem302} %\blue{[moved here]}
For fixed $b\in\bR$, we have $\lim_{b^\prime\to b} T_{b^\prime}=T_b$ on $\{T_b < \infty\}$, almost surely. 
\end{lemma}

\begin{proof}[Proof of Proposition \ref{Lem404}]
%\blue{[I changed the proof substantially.]}
    %\blue{[moved here] 
By Lemma \ref{Lem301}, we can decompose the expected costs {as follows:}
\begin{align}
v_b (x) = v_b^{(1)}(x)  + C v_b^{(2)}(x),  \label{def_v_b}
\quad x\in\R, 
\end{align}
where we write
\begin{align}
v_b^{(1)}(x) := \E_x\left[ \int_0^\infty e^{-qt} f(U^b_t)\diff t   \right], \quad
v_b^{(2)}(x) := \E_x\left[ \int_0^\infty e^{-qt}  \diff R^b_t  \right]. 
\end{align}
%  This decomposition makes sense  }
  %\eqref{def_v_b} makes sense by the following lemma whose proofs are given in Appendix \ref{appendix_proof}.}

%\blue{[changed $(y)$ to $[y]$ because $T^{(n)}_b$ has been used.]}
For $y\in\bR$, we write $X^{[y]}_t:= X_t +y$, $t\geq 0$, and write $U^{[y],b}$, $R^{[y],b}_t$, $T^{[y]}_b$ be those corresponding to this shifted process.

(i)
Fix $b \in \R$ and $\varepsilon > 0$. We show that $t \mapsto U^{[\varepsilon], b}_t-U^b_t= \varepsilon+ R^{[\varepsilon], b}_t-R^b_t$ is nonincreasing and {always} lies on $[0, \varepsilon]$. Because this difference is a step function {in $t$} with jump times contained in the set {$\mathcal{T}_\eta$,}
%$\{ T(k) \}_{k \in \mathbb{N}}$, 
it suffices to show
\[
\zeta(k) := U^{[\varepsilon],b}_{T(k)}-U^b_{T(k)}=\varepsilon + R^{[\varepsilon],b}_{T(k)}-R^b_{T(k)}, \quad k \geq 0,
\]
is nonincreasing in $k$ and takes values only on $[0,\varepsilon]$.
%
%
%Consider the difference $U^{b+\varepsilon}_t-U^b_t=R^{b+\varepsilon}_t-R^b_t$. We shall show the following claim.
%\begin{enumerate}
%\item it increases.
%\item between $[0, \varepsilon]$
%\item zero on $[0, T_b+\varepsilon)$ 
%\item $\varepsilon$ on $[T_{b}, \infty)$.
%\end{enumerate}
%The third claim is immediate.
%we can only focus on these times.
We show this claim by induction. 

First it holds trivially when $k = 0$ with $\zeta(0) = {\varepsilon} \in [0,\varepsilon]$.

Now, suppose it holds that $\zeta(k) \in [0,\varepsilon]$ for some $k \geq 0$.
%
%%at $T(k)$ with $k \in \bN \cup \{0\}$, we have 
%\begin{align}
%\begin{aligned}
%U^{b+\varepsilon}_{T(k)}-U^{b}_{T(k)}=R^{b+\varepsilon}_{T(k)}-R^{b}_{T(k)}\in[0, \varepsilon].
%% \quad U^{b}_{T(k)} \geq b
%%, \quad U^{b+\varepsilon}_{T(k)}\geq b+\varepsilon.
%%R^{b+\varepsilon}_{T(k)}-R^{b}_{T(k)}=U^{b+\varepsilon}_{T(k)}-U^{b}_{T(k)}, \quad R^{b}_{T(k)}\geq 0. 
%\end{aligned}\label{6}
%\end{align}
With the set of indices of controlling: %\blue{[here changed from $\varepsilon$ to $\delta$ because it is used for the case $\delta = 0$]}
\[
A^{[\delta]} := \{ k \geq 1: \Delta R_{T(k)}^{[\delta], b} > 0\} = \{ k \geq 1: U^{[\delta],b}_{T(k-1)}+(X_{T(k)} - X_{T(k-1)}) < b \}, \quad \delta = 0, \varepsilon,
% \tilde{U}^{b}_{T(k)} \neq U^{b}_{T(k)}\} = \{ k \geq 0: \tilde{U}^{b}_{T(k)} < b \}.
\]
we have
\begin{multline}
\{ k+1 \in A^{[\varepsilon]} \} %= \{\tilde{U}^{b}_{T(k+1)} < b \} 
= \{ U^{b}_{T(k)}  + \zeta(k) +(X_{T(k+1)} - X_{T(k)}) < b \} \\ \subset  \{ U^{b}_{T(k)}+(X_{T(k+1)} - X_{T(k)}) < b \} = \{ k+1 \in A^{[0]} \}. \label{set_inclusion2}
\end{multline}
%In addition,
%\[
%\{ k+1 \notin A_b, k+1 \notin A_{b+\varepsilon} \} 
%= \{ U^{b}_{T(k)}+(X_{T(k+1)} - X_{T(k)}) > b, U^{b+\varepsilon}_{T(k)}+(X_{T(k+1)} - X_{T(k)}) > b + \varepsilon \}
%\]
%and 
%\[
%\{ k+1 \notin A_b, k+1 \in A_{b+\varepsilon} \} 
%= \{ U^{b}_{T(k)}+(X_{T(k+1)} - X_{T(k)}) > b, U^{b+\varepsilon}_{T(k)}+(X_{T(k+1)} - X_{T(k)}) < b + \varepsilon \}
%\]

(A) Suppose $k+1 \in A^{[\varepsilon]}$ so that $U^{[\varepsilon],b}_{T(k+1)} = b$. By \eqref{set_inclusion2}, this also implies $k+1 \in A^{[0]}$ or equivalently $U^{b}_{T(k+1)} = b$. Hence, $\zeta(k+1) = 0$.

(B) Suppose $k+1 \notin A^{[\varepsilon]}$ so that 
\begin{align}
U^{[\varepsilon],b}_{T(k)}+(X_{T(k+1)} - X_{T(k)})  \geq b. \label{U_bigger_b2}
\end{align}

(a) Suppose $k+1 \notin A^{[0]}$, then because $\Delta R^{[\varepsilon],b}_{T(k+1)} = \Delta R^{b}_{T(k+1)} = 0$, we have $\zeta(k+1) = \zeta(k) \in [0,\varepsilon]$.

(b) Suppose $k+1 \in A^{[0]}$. Then, clearly $\zeta(k+1) = \zeta(k) - \Delta R^{b}_{T(k+1)} < \zeta(k)$. In addition, by \eqref{U_bigger_b2},
\[
 \zeta(k+1) = (U^{[\varepsilon],b}_{T(k)}+(X_{T(k+1)} - X_{T(k)})) - b \geq 0.
\]
In sum, in all cases we have $\zeta(k+1) \leq \zeta(k)$ and in addition, $\zeta(k+1) \in [0, \varepsilon]$. By mathematical induction we have that $\zeta$ is nonincreasing and always lies {in} $[0, \varepsilon]$. 
In view of \eqref{set_inclusion2}, this also shows $A^{[\varepsilon]} \subset A^{[0]}$. 

%\blue{[need to say that control cannot happen when the process is right at zero.]}

At the moment %$T_b = \inf_{k \in A^{(0)}} T(k)$
{$T^{[\varepsilon]}_b = \inf_{k \in A^{[\varepsilon]}} T(k)$} {with $\inf \varnothing = \infty$}, %\blue{[changed back to inf because the set may be empty]} 
the difference {between $U^{[\varepsilon], b}$ and $U^{b}$} becomes $0$ and must stay at $0$ afterwards. On the other hand, before %$T^{(\varepsilon)}_{b}$
{$T_{b}$} there is no control for both and the difference is $\varepsilon$. In sum, %\blue{[below changed from $T^{(0)}_b$ to $T_b$]}
\begin{align}
U^{[\varepsilon], b}_t-U^b_t=
\begin{cases}
\varepsilon, \quad &t\in[0, T_b), \\
0, \quad&t\in[{T^{[\varepsilon]}_b}, \infty),
\end{cases}
\quad 
R^{[\varepsilon], b}_t-R^b_t=
\begin{cases}
0, \quad &t\in[0, T_b), \\
-\varepsilon, \quad&t\in[{T^{[\varepsilon]}_b}, \infty).
\end{cases}
\label{14}
\end{align}
\par
(ii) 
%We prove, for $x\in\bR$, 
%\begin{align}
%\lim_{\varepsilon\downarrow0}\frac{v^{(1)}_b(x+\varepsilon)-v^{(1)}_b(x)}{\varepsilon}=\bE_x \left[\int_0^{T_b} e^{-qt} f^\prime_+(X_t) \diff t \right], \label{17}\\
%\lim_{\varepsilon\downarrow0}\frac{v^{(1)}_b(x)-v^{(1)}_b(x-\varepsilon)}{\varepsilon}=\bE_x \left[\int_0^{T_b} e^{-qt} f^\prime_-(X_t) \diff t \right].\label{18}
%\end{align}
%We have 
%\begin{align}
%\frac{v^{(1)}_b(x+\varepsilon)-v^{(1)}_b(x)}{\varepsilon}
%=\bE_x \left[ \int_0^\infty e^{-qt} \frac{f(U^{(\varepsilon),b}_t )-f(U^{(0),b}_t ) }{\varepsilon}\diff t\right].
%\end{align}
By \eqref{14}, we have
%By (i) (in particular that the process $\{U^{(\varepsilon),b}_t-U^b_t : t\geq 0\}$ is nonincreasing and from \eqref{14}), %\blue{[I guess we are only using \eqref{14}?]}
\begin{align}
\frac{v^{(1)}_b(x+\varepsilon)-v^{(1)}_b(x)}{\varepsilon} =& \bE_x \left[    \int_0^{T_b} e^{-qt} \frac{f(U^{b}_t + \varepsilon)-f(U^{b}_t ) }{\varepsilon}\diff t \right] \\
&+ \bE_x \left[   \int_{T_b}^{T_b^{[\varepsilon]}} e^{-qt} \frac{f(U^{[\varepsilon],b}_t )-f(U^{b}_t ) }{\varepsilon}\diff t  \right].
\end{align}
By (i) (in particular that the process $\{U^{[\varepsilon],b}_t-U^b_t : t\geq 0\}$ is nonincreasing), mean value theorem, and the convexity of $f$, for all $0 < \varepsilon < \bar{\varepsilon}$,
\begin{align}
\left| \bE_x \left[   \int_{T_b}^{T_b^{[\varepsilon]}} e^{-qt} \frac{f(U^{[\varepsilon],b}_t )-f(U^{b}_t ) }{\varepsilon}\diff t  \right] \right| & \leq \bE_x \left[   \int_{T_b}^{T_b^{[\varepsilon]}} e^{-qt} \frac{|f(U^{[\varepsilon],b}_t )-f(U^{b}_t )| }{\varepsilon}\diff t  \right] \\
 \leq \bE_x &\left[ \int_{T_b}^{T_b^{[\varepsilon]}} e^{-qt} \sup_{y \in [U_t^b, U_t^b+\bar{\varepsilon}]} |f'_+(y)|
 %\frac{|f(U^{b+\varepsilon}_t)-f(U^b_t)|}{\varepsilon}
 \diff t\right]  \\
 \leq \bE_x &\left[\int_{T_b}^{T_{b-\varepsilon}}  e^{-qt} (|f'_+(U_t^b)| + |f'_+(U_t^b+\bar{\varepsilon})|)
 %\frac{|f(U^{b+\varepsilon}_t)-f(U^b_t)|}{\varepsilon}
 \diff t\right] \\
 \xrightarrow{\varepsilon \downarrow 0}& 0,
\end{align}
where 
{$T_b^{[\varepsilon]} = T_{b-\varepsilon}$ holds because $\{t < T_b^{[\varepsilon]} \} = \{ \min_{1 \leq k \leq N_t^\eta} X_{T(k)}^{[\varepsilon]} \geq b
\}  = \{ \min_{1 \leq k \leq N_t^\eta} X_{T(k)} \geq b-\varepsilon
\} = \{t < 
{T_{b-\varepsilon}}
\}$
%T_b^{(\varepsilon)} \}$ 
(see Remark \ref{remark_simple_expression_R}) and}
the last limit holds by monotone convergence and Lemma \ref{Lem302}. Note that the finiteness of the expectations above hold by {Corollary} \ref{Lem303}.
%
%\begin{align}
%\begin{aligned}
% \bE_x \left[\inf_{s\in[T_{b+\varepsilon}, T_b] }\int_s^\infty e^{-qt} \frac{f(U^{b+\varepsilon}_t)-f(U^b_t)}{\varepsilon}\diff t\right]&\leq
%\frac{v^{(1)}_{b+\varepsilon}(x)-v^{(1)}_b(x)}{\varepsilon}\\
%\leq &\bE_x \left[\sup_{s\in[T_{b+\varepsilon}, T_b] }\int_s^\infty e^{-qt} \frac{f(U^{b+\varepsilon}_t)-f(U^b_t)}{\varepsilon}\diff t\right]. 
%\end{aligned}\label{9}
%\end{align}
%By \eqref{7},
%\begin{align}
%&\left|\bE_x \left[\sup_{s\in[T_{b+\varepsilon}, T_b] }\int_s^\infty e^{-qt} \frac{f(U^{b+\varepsilon}_t)-f(U^b_t)}{\varepsilon}\diff t\right]-\bE_x \left[\inf_{s\in[T_{b+\varepsilon}, T_b] }\int_s^\infty e^{-qt} \frac{f(U^{b+\varepsilon}_t)-f(U^b_t)}{\varepsilon}\diff t\right] \right|\\
%\label{8}&\qquad \leq \bE_x \left[\int_{T_{b+\varepsilon}}^{T^{(\varepsilon)} \blue{T_b??}} e^{-qt} \frac{|f(U^{b+\varepsilon}_t)-f(U^b_t)|}{\varepsilon}\diff t\right]
%\xrightarrow{\varepsilon \downarrow 0} 0,
%\end{align}
%%as $\varepsilon \downarrow 0$ 
%where the last limit holds 
%by Lemmas \ref{Lem302} and \ref{Lem303} and the dominated convergence theorem. 
%From \eqref{9}, \eqref{8}, \eqref{7} and Lemma \ref{Lem303} and the dominated convergence theorem, we have 
Now by the convexity of $f$, monotone convergence gives
%\begin{align}
%\lim_{\varepsilon\downarrow 0}\frac{v^{(1)}_{b+\varepsilon}(x)-v^{(1)}_b(x)}{\varepsilon}
%=&\lim_{\varepsilon\downarrow 0}\bE_x \left[\int_{T_b}^\infty e^{-qt} \frac{f(U^{b}_t+\varepsilon)-f(U^b_t)}{\varepsilon}\diff t\right]
%= \bE_x \left[\int_{T_b}^\infty e^{-qt} f^\prime_+(U^b_t) \diff t \right]. \label{11}
%\end{align}
%
%
%
%From \eqref{14}, we have 
%\begin{align}
%\begin{aligned}
% \bE_x \left[\inf_{s\in[T^{(0)}_{b}, T^{(\varepsilon)}_b] }\int_0^s e^{-qt} \frac{f(U^{(\varepsilon),b}_t)-f(U^{(0),b}_t)}{\varepsilon}\diff t\right]&\leq
%\frac{v^{(1)}_b(x+\varepsilon)-v^{(1)}_b(x)}{\varepsilon}\\
%\leq &\bE_x \left[\sup_{s\in[T^{(0)}_{b}, T^{(\varepsilon)}_b] }\int_0^s e^{-qt} \frac{f(U^{(\varepsilon),b}_t)-f(U^{(0),b}_t)}{\varepsilon}\diff t\right]. 
%\end{aligned}\label{9}
%\end{align}
%We have 
%\begin{align}
%&\left|\bE_x \left[\sup_{s\in[T^{(0)}_{b}, T^{(\varepsilon)}_b] }\int_0^s e^{-qt} \frac{f(U^{(\varepsilon),b}_t)-f(U^{(0),b}_t)}{\varepsilon}\diff t\right]-\bE_x \left[\inf_{s\in[T^{(0)}_{b}, T^{(\varepsilon)}_b] }\int_0^s e^{-qt} \frac{f(U^{(\varepsilon),b}_t)-f(U^{(0),b}_t)}{\varepsilon}\diff t\right] \right|\\
%\label{8}&\qquad \leq \bE_x \left[\int_{T^{(0)}_{b}}^{T^{(\varepsilon)}_b} e^{-qt} \frac{|f(U^{(\varepsilon),b}_t)-f(U^{(0),b}_t)|}{\varepsilon}\diff t\right]
%\to0,
%\end{align}
%as $\varepsilon \downarrow 0$ by Lemma \ref{Lem302}, \ref{Lem303} and the dominated convergence theorem. 
%From \eqref{9}, \eqref{8}, \eqref{7} and Lemma \ref{Lem303} and the dominated convergence theorem, we have 
\begin{align}
\lim_{\varepsilon\downarrow 0}\frac{v^{(1)}_b(x+\varepsilon)-v^{(1)}_b(x)}{\varepsilon}
=&\lim_{\varepsilon\downarrow 0}\bE_x \left[\int_0^{T_b} e^{-qt} \frac{f(U^{b}_t+\varepsilon)-f(U^b_t)}{\varepsilon}\diff t\right]\\
=&\bE_x \left[\int_0^{T_b} e^{-qt} f^\prime_+(U^b_t) \diff t \right]. 
\end{align}
%which implies \eqref{17}.
%By changing from $x$ to $x-\varepsilon$ in the above argument, we have 
%{[How about the following?]

In the same way, we compute the left derivative. 
By (i) with $x$ changed to $x-\varepsilon$, we have 
\begin{align}
\frac{v^{(1)}_b(x)-v^{(1)}_b(x-\varepsilon)}{\varepsilon} = \bE_x \left[    \int_0^{T_b} e^{-qt} \frac{f(U^{b}_t )-f(U^{b}_t -\varepsilon) }{\varepsilon}\diff t \right] + {h(\varepsilon)},
%-\bE_x \left[    \int_{T^{(-\varepsilon)}_b}^{T_b} e^{-qt} \frac{f(U^{b}_t )-f(U^{b}_t -\varepsilon) }{\varepsilon}\diff t \right]
%\\
%&+ \bE_x \left[   \int_{T_b^{(-\varepsilon)}}^{T_b} e^{-qt} \frac{f(U^{b}_t )-f(U^{(-\varepsilon),b}_t ) }{\varepsilon}\diff t  \right].
\end{align}
{where 
\[
h(\varepsilon) :=  \bE_x \left[   \int_{T_b^{[-\varepsilon]}}^{T_b} e^{-qt} \frac{f(U^{b}_t )-f(U^{[-\varepsilon],b}_t ) }{\varepsilon}\diff t  \right]-\bE_x \left[    \int_{T^{[-\varepsilon]}_b}^{T_b} e^{-qt} \frac{f(U^{b}_t )-f(U^{b}_t -\varepsilon) }{\varepsilon}\diff t \right].
\]}
For all $0 < \varepsilon < \bar{\varepsilon}$, {mean value theorem and the convexity of $f$ give}
\begin{align}
\left| \frac{f(U^{b}_t )-f(U^{b}_t -\varepsilon) }{\varepsilon} \right| \lor
\left| \frac{f(U^{b}_t )-f(U^{[-\varepsilon],b}_t ) }{\varepsilon}\right| \leq 
%\sup_{y \in [U_t^b-\bar{\varepsilon} , U_t^b]} |f'_+(y)| 
|f'_+(U_t^b-\bar{\varepsilon})| + |f'_+(U_t^b)|, \quad t \geq 0,
\end{align}
and thus
%\begin{align}
%&\bE_x \left[    \int_{T^{(-\varepsilon)}_b}^{T_b} e^{-qt} \frac{f(U^{b}_t )-f(U^{b}_t -\varepsilon) }{\varepsilon}\diff t \right]
%+ \bE_x \left[   \int_{T_b^{(-\varepsilon)}}^{T_b} e^{-qt} \frac{f(U^{b}_t )-f(U^{(-\varepsilon),b}_t ) }{\varepsilon}\diff t \right]\\
%&\leq 2\bE_x \left[ \int_{T^{(-\varepsilon)}_b}^{T_b} e^{-qt} \sup_{y \in [U_t^b-\bar{\varepsilon}, U_t^b]} |f'_+(y)| \diff t\right] 
${|h(\varepsilon)|}\leq 2\bE_x \left[\int_{ {T_{b+\varepsilon}}}^{T_b}  e^{-qt} (|f'_+(U_t^b-\bar{\varepsilon})| + |f'_+(U_t^b)|)
 \diff t\right] \xrightarrow{\varepsilon \downarrow 0} 0$ 
 where {we used} $T^{[-\varepsilon]}_b=T_{b+\varepsilon}$ {and Lemma \ref{Lem302}}.
%\end{align}
Therefore, as {in} the case of right derivative, we have {by monotone convergence,}
\begin{align}
\lim_{\varepsilon\downarrow 0}\frac{v^{(1)}_b(x)-v^{(1)}_b(x-\varepsilon)}{\varepsilon}
=&\lim_{\varepsilon\downarrow 0}\bE_x \left[\int_0^{T_b} e^{-qt} \frac{f(U^{b}_t)-f(U^b_t-\varepsilon)}{\varepsilon}\diff t\right]\\
=&\bE_x \left[\int_0^{T_b} e^{-qt} f^\prime_-(U^b_t) \diff t \right]. 
\end{align}
%Since the potential measure of $X$ has no atoms {by Assumption \ref{Ass202}(1) and \cite[Proposition I.15]{Ber1996}}, 
{Because the right and left derivatives coincide thanks to Remark \ref{remark_mass} and $U^{b}_t = X_t$ for $t < T_b$,} %we obtain 
\[v_b^{(1) \prime} (x)%\lim_{\varepsilon \to 0}\frac{v_b(x+\varepsilon)-v_b(x)}{\varepsilon}
=\bE_x \left[\int_0^{T_b} e^{-qt} f^\prime_+(U_t^b) \diff t \right]
=\bE_x \left[\int_0^{T_b} e^{-qt} f^\prime_+(X_t) \diff t \right].
\]
%which implies \eqref{18}. 

\par
(iii) %We compute the right-derivative of the map $b\mapsto v_b^{(2)}(x)$ for fixed $x\in\bR$ to show
We now show
\begin{align}
\lim_{\varepsilon\downarrow0}\frac{v^{(2)}_{b}(x+\varepsilon)-v^{(2)}_b(x)}{\varepsilon}=
\lim_{\varepsilon\downarrow0}\frac{v^{(2)}_{b}(x)-v^{(2)}_b(x-\varepsilon)}{\varepsilon}= -\bE_x \left[e^{-qT_b}\right]. \label{16}
\end{align}
%\blue{[below changed $R^{(0), b}$ to $R^{b}$]} 
Indeed, since the process $\{R^{[\varepsilon],b}_t-R^{b}_t : t\geq 0\}$ is nonincreasing by (i) and from \eqref{14}, we have 
\begin{align}
\begin{aligned}
- \bE_x \left[ e^{-qT_b}\right]\leq
\frac{v^{(2)}_{b}(x+\varepsilon)-v^{(2)}_b(x)}{\varepsilon}
\leq  -\bE_x \left[ e^{-qT_{b+\varepsilon}}\right]. 
\end{aligned}\label{15}
\end{align}
By Lemma \ref{Lem302} and \eqref{15}, we have that the first term in \eqref{16} is equal to the third term in \eqref{16}. 
By changing from $x$ to $x-\varepsilon$ in the above argument, we have the second equality in \eqref{16}. 
\par
From (ii) and (iii), we obtain \eqref{20}. 
\par
(iv) {It remains to show that {$x \mapsto v^\prime_b(x)$} is continuous.} %belongs to $C(\bR)$. 
We have 
\begin{align}
&\left|v_{b}^\prime (x+\varepsilon)-v_{b}^\prime (x)\right|\\
&\quad \leq 
\left|\bE_{x+\varepsilon} \left[\int_0^{T_b} e^{-qt} f^\prime_+(X_t) \diff t \right]
-\bE_x \left[\int_0^{T_b} e^{-qt} f^\prime_+(X_t) \diff t \right]\right|\\&\quad\quad
+|C|\left|\bE_{x+\varepsilon} \left[ e^{-qT_b}\right]-\bE_x \left[ e^{-qT_b}\right]\right|\\
&\quad= 
\left|\bE_{x} \left[\int_0^{T_{b-\varepsilon}} e^{-qt} f^\prime_+(X_t+\varepsilon) \diff t \right]
-\bE_x \left[\int_0^{T_b} e^{-qt} f^\prime_+(X_t) \diff t \right]\right|\\&\quad\quad
+|C|\left|\bE_{x} \left[ e^{-qT_{b-\varepsilon}}\right]-\bE_x \left[ e^{-qT_b}\right]\right|
\\
%&\leq
%\bE_x \left[ \int_0^{T_b} e^{-qt} \left| f^\prime_+(X_t+\varepsilon)- f^\prime_+(X_t) \right| \diff t \right]
%+\bE_x \left[ \int_{T_{b-\varepsilon}}^{T_b} e^{-qt} \left|  f^\prime_+(X_t) \right| \diff t \right]\\
&\quad\leq
\bE_x \left[ \int_0^{T_b} e^{-qt} \left| f^\prime_+(X_t+\varepsilon)- f^\prime_+(X_t) \right| \diff t \right]
+\bE_x \left[ \int_{T_b}^{T_{b-\varepsilon}} e^{-qt} \left|  f^\prime_+(X_t+\varepsilon) \right| \diff t \right]\\
&\qquad \quad +|C|\left|\bE_{x} \left[ e^{-qT_{b-\varepsilon}}\right]-\bE_x \left[ e^{-qT_b}\right]\right|.
\end{align}
{As $\varepsilon \searrow 0$,} the first expectation converges to zero by monotone convergence, the second expectation converges to zero by monotone convergence and Lemma \ref{Lem302}. The last expectation converges to zero by Lemma \ref{Lem302}. 
%\blue{By replacing $\varepsilon$ with $-\varepsilon$, we also have the left continuity. [CHECK]}
By replacing $x$ with $x-\varepsilon$ and again using 
{Remark \ref{remark_mass},} 
%the fact that the potential measure of $X$ has no atoms,} 
we also have the left continuity.
\end{proof}

%We define {the first control time under the policy $\pi^b$:}
%stopping times as follows: for $b, x \in\bR$, %\blue{[delete $x$?]}
%\begin{align}
%T_b:= \inf\{t\geq 0 :  R^{b}_t>0  \}. \label{first_control_time}
%\end{align}
%\blue{[actually this is the same as $T_b^{(1)}$. Better to just keep using $T_b^{(1)}$?]}
%{[You are right. How about the following?]
%{Let us write the first control time under the policy $\pi^b$ as $T_b := T_b^{(1)}$.}
%{[Should we delete the sentence that follows since it is clear from the definition of $T_b^{(1)}$?]}

\section{The optimal barrier $b^\ast$ in the periodic barrier strategies} \label{section_optimality}
In this section, we {show the optimality of a periodic barrier strategy.}
% for a suitable selection of the barrier.}
%give the optimal barrier values of the periodic barrier strategies. 
%For this, we give some lemmas. The proofs of Lemmas \ref{Lem303}--\ref{Lem302} are given in Appendix \ref{appendix_proof}. 
%Let we define the function $\rho$ as
Define 
\begin{align}
\rho(b):=\bE_b \left[\int_0^\infty e^{-qt} f^\prime_+ (U^b_t) \diff t \right],\qquad b\in\bR, \label{def_rho}
\end{align}
which takes real values by {Corollary} \ref{Lem303}.  Our candidate optimal barrier is
%\begin{align}
%b^\ast:= \inf\{b\in\bR:\rho (b) +C\geq 0\}
%\end{align}
%where
\begin{align}
b^\ast:= \inf\{b\in\bR:\rho (b) +C\geq 0\}, %\quad \textrm{where} \quad \rho(b) := \bE_b \left[\int_0^\infty e^{-qt} f^\prime_+ (U^b_t) \diff t \right], 
\label{33}
\end{align} 
%\blue{[above: maybe ok to change $\geq$ to $=$?]}
%{[Under the assumptions in this case, I think that's fine.
%However, in the case of driftless compound Poisson, I think the original definition should be used.
%Considering expansion, I think it's fine as is.]}
which is well-defined by Lemma \ref{Lem401} below;
%The function $\rho$ 
%\blue{is used to chracterize} the optimal barrier value and has the following properties; 
{see Appendix \ref{proof_lem401} for the proof. } 
\begin{lemma}\label{Lem401}
The function $\rho$ is nondecreasing and continuous. 
In addition, we have $\lim_{b \uparrow\infty} \rho(b)= f^\prime_+(\infty) /q > {-C}$ and $\lim_{b \downarrow-\infty} \rho(b)= f^\prime_+(-\infty) /q < {-C}$. 
\end{lemma}

%which is non-decreasing since $f^\prime_+$ is so. 
%Note that {$b^\ast$ is well defined and finite} by Lemma \ref{Lem401}.%since $\lim_{b \uparrow\infty} \rho(b)= f^\prime_+(\infty) /q > 0$ and $\lim_{b \downarrow-\infty} \rho(b)= f^\prime_+(-\infty) /q{< 0}$ by Assumption \ref{Ass201} (3). 

%\blue{[probably we can combine the previous two lemmas and make a lemma saying (1) $\rho$ is finite, (2) it is increasing, (3) it is continuous.]}

%\subsection{Slope of $v_{b^*}$} 
%\blue{We shall now obtain the slope of the candidate value function $v_{b^*}$, which is important for the verification of optimality.}

%\subsection{Verification} \blue{[changed to subsection]}
%\blue{[added this.]}
%In this section, we put the verification lemma, which give the sufficient condition for an optimal strategy, 
%and prove that the periodic barrier strategy at $b^\ast$ satisfies the conditions in the verification lemma. 
%The main theorem of this paper is the following: 
{We now state the main result of the paper.}
\begin{theorem}\label{Thm501}
%{Under Assumptions \ref{Ass201}, \ref{Ass202} and \ref{Ass501} (which will be given later), } \blue{[let me think how to present]}
The periodic barrier strategy at $b^\ast$ is an optimal strategy and thus we have 
$v(x)=v_{b^\ast} (x)$ for $x\in\bR$. 
\end{theorem}
\par
{In the remaining {part}, we show Theorem \ref{Thm501}.}
%\subsection{Optimality result} %\blue{[we should not call this a special case, because the assumption is very minor.]}
%\blue{[ok to delete the first two sentences and just start with ``Let...''?]}
%We prove Theorem \ref{Thm501} under an assumption. 
%Before giving the assumption, we define the infinitesimal generator. 
%Let $\cL$ be the operator applied {to} a 
{Acting on a measurable function  $g: \bR\to\bR$ belonging to}
%polynomial growth function $g$ which 
$C^1(\bR)$ (resp., $C^2(\bR)$) when $X$ has bounded (resp., unbounded) variation paths {with at most polynomial growth, define {the} operator}%defined by 
\begin{align}
\cL g(x) := \gamma g^\prime (x) 
+\frac{1}{2}\sigma^2 g^{\prime \prime}(x) 
+\int_{\R \backslash \{0\} }  (g(x+z)-g(x) -g^\prime (x) z1_{\{|z|<1\}}) \Pi(\diff z), \quad   x \in \R.
\end{align}
{Let $(\mathcal{L} - q) g := \mathcal{L} g - q g$.}
%{Note that $\cL$ is well defined for $g$ since $g$ is polynomial growth and by Assumption \ref{Ass202}(2)} \blue{[it is proved later and so we can delete this sentence.]}
%\blue{[moved here]}
Define also for any measurable function $g: \bR\to\bR$, 
\begin{align}
\cM g(x)&:= \inf_{l\geq 0}\{Cl +g(x+l)\},  \quad x \in \R.
\end{align}

The following verification lemma gives a sufficient condition for optimality. The proof is   the same as that for the spectrally negative case in \cite[Lemma 3.1]{Perez_Yamazaki_Bensoussan} {and hence we omit {it}.} 

\begin{lemma}[verification lemma] \label{Lem501}
Let $w: \bR \to \bR$ be of polynomial growth and {belongs to $C^1(\bR)$ (resp., $C^2(\bR)$) when $X$ has bounded (resp., unbounded) variation paths.} 
If it satisfies
\begin{align}
(\cL -q)w(x)+ {\eta} (\cM w(x) -w (x)) + f(x)=0, \quad x\in\bR, \label{19}
\end{align}
then, we have $w(x) \leq v(x)$ for $x\in\bR$. 
\end{lemma}

{Before confirming these sufficient conditions for $w = v_{b^*}$, we explicitly compute $\mathcal{M} v_{b^*}$.}
{To this end, we show the following, which is a direct consequence of Proposition \ref{Lem404}.}
% we have the slope of $v_{b^*}$ written in terms of the process $U^{b^*}$.}
\begin{lemma}\label{Lem405}
%\blue{[changed from Lemma to Proposition]}
For $x\in\bR$, we have 
%\begin{align}
$v_{b^\ast}^\prime (x) = \bE_x \left[\int_0^\infty e^{-qt} f^\prime_+ (U^{b^\ast}_t) \diff  t\right].$
%\end{align}
\end{lemma}
\begin{proof}
%  by Lemma \ref{Lem404}, 
%\begin{align*}
%v_{b^\ast}^\prime (x)=\bE_x \left[\int_0^{T_{b^\ast}} e^{-qt} f^\prime_+(X_t) \diff t \right]-C\bE_x \left[ e^{-qT_{b^\ast}}\right]
%=\bE_x \left[\int_0^{T_{b^\ast}} e^{-qt} f^\prime_+(X_t) \diff t \right]+\bE_x \left[ e^{-qT_{b^\ast}}\right]\rho (b^\ast).
%\end{align*}
%\blue{[shortened the proof.]}
By Lemma \ref{Lem401} and the definition of $b^\ast$, we have $\rho (b^\ast)+C=0$. This together with the  strong Markov property gives
$-C\bE_x \left[ e^{-qT_{b^\ast}}\right] = \bE_x \left[ e^{-qT_{b^\ast}}\right]\rho (b^\ast) =
\bE_x [\int_{T_{b^\ast}}^\infty e^{-qt} f^\prime_+(U^{b^\ast}_t) \diff t ]$ {where we recall $T_{b^\ast} := T_{b^\ast}^{(1)}$ is the first control time under the policy $\pi^{b^\ast}$. }
Subsituting this in \eqref{20} gives the result.
\end{proof}
%{This verifies that $v_{b^*}$ is convex and hence implies the following.}
%First we write $\mathcal{M} v_{b^*}$ explicitly as follows.

%\blue{[moved the following outside of the proof] }
From (i) of the proof of Proposition \ref{Lem404}, {for each $t \geq 0$,} $U_t^{b^*}$ is monotonically increasing in the start value $X_0 = x$. By this and Lemma \ref{Lem405}, the {derivative $v_{b^\ast}'$} is nondecreasing. In addition, by the definition of $b^\ast$ {and the continuity of $\rho$ as in Lemma \ref{Lem401}}, we have $v_{b^\ast}^\prime(b^\ast)=-C$. 
Thus, we have 
\begin{align}
v_{b^\ast}^\prime (x)
\begin{cases}
\leq -C, \quad &x< b^\ast, \\
\geq -C, \quad &x\geq b^\ast.
\end{cases}
\end{align}
Since the derivative of the function $l \mapsto C l + v_{b^\ast} (x+l)$ is equal to $l\mapsto C+v_{b^\ast}^\prime (x+l)$, it is minimized when $l = (b^* - x)^+$, {showing the following.} %\blue{In sum, we have the following.}

\begin{proposition} \label{prop_M} We have
%\begin{align}
%$\cM v_{b^\ast}(x)= {C (b^\ast - x )^+  + v_{b^*}(x \vee b^*),}$ {for all $x \in \R$.}
{
\[
\cM v_{b^\ast}(x)= \left\{ \begin{array}{ll}
v_{b^*}(x), & x \geq b^*, \\ C (b^\ast - x )  + v_{b^*}(b^*), & x < b^*. 
\end{array}
\right.
\]
}
% \label{23}
%\begin{cases}
%C(b^\ast - x )+ v_{b^\ast} (b^\ast)\quad &x < b^\ast, \\
%v_{b^\ast}(x)\quad & x\geq b^\ast.
%\end{cases}
%\end{align}
\end{proposition}
%\begin{proof}%[Proof of Proposition \ref{prop_M}]
%%We compute $\cM v_{b^\ast}$. 
%
%\end{proof}
}

%\par
%We prove that $v_{b^\ast}$ is sufficiently smooth and satisfies \eqref{19}. 

{Regarding the smoothness of $v_{b^*}$, it belongs to $C^1(\bR)$ by {Proposition} \ref{Lem404}. This is sufficient for the case of bounded variation but a care is needed for the unbounded variation case.  We temporarily assume the following, {to first consider the case {the} $C^2$ property of $v_{b^*}$ is guaranteed.}
%In the case of unbounded variation, it is guaranteed to belong to $C^2(\bR)$ when $f \in C^{2}(\bR)$.
\begin{cond}\label{Ass501}
When $X$ has unbounded variation paths, {the running cost function} $f$ belongs to $C^{2}(\bR)$ and $f^{\prime \prime}$ has polynomial growth in the {tails}. %\blue{[maybe we can relax this later?]}{[I guess we can relax as classical cases.]}
\end{cond}
{The proof of the following lemma is deferred to Appendix \ref{proof_Lem502}.}
\begin{lemma}\label{Lem502} %\blue{[I combined the two lemmas.]}
{Suppose Condition \ref{Ass501} holds.}
When $X$ has unbounded variation paths, the function $v_{b^\ast}$ belongs to $C^2(\bR)$. 
\end{lemma}
%In the remaining arguments, we assume 
We assume Condition \ref{Ass501} temporarily for Lemma \ref{Lem503}. However, Condition \ref{Ass501} can be completely relaxed by following the arguments in Section 4.2 in \cite{NobYam2020}. We later provide a brief remark on how {Condition \ref{Ass501}} can be {removed} in the proof of Theorem \ref{Thm501}.
%the value function of a general case can be approximated by those of the cases 
}
%
%In the case of bounded variation, $v_{b^\ast}$ belongs to $C^1(\bR)$ by \blue{Proposition} \ref{Lem404}. 
%In the case of unbounded variation, we give the following assumption. 
%
%{In this paper, we prove the main result under Assumption \ref{Ass501}. Assumption \ref{Ass501} can be relaxed, but we omit the details of this in this paper, since it is just the same argument as for Section 4.2 in \cite{NobYam2020}. 
%%The main result holds without . We can show this by making exactly the same argument as we did for Section 4.2 in \cite{NobYam2020}.
%}
%%Under this assumption, we have the following lemma. 
%
%\blue{[moved the proof to the appendix]}

By {Proposition} \ref{Lem404} and {Lemma} \ref{Lem502}, the function ${x \mapsto v_{b^\ast}(x)}$ is sufficiently smooth to apply $\cL$ {(under Condition \ref{Ass501})}. 
\begin{lemma}\label{Lem503} {Suppose Condition \ref{Ass501} holds.}
For $x \in\bR$, we have 
%\begin{align}
%\cL v_{b^\ast}(x) - (q+\eta)v_{b^\ast}(x) + h(x)=0, \label{21}
%\end{align}
%we have
$(\cL -q)v_{b^\ast}(x)+{\eta} (\cM v_{b^\ast}(x) -v_{b^\ast} (x)) + f(x)=0$.
%where %\blue{[below ok to change $y$ to $x$?]}
\end{lemma}
\begin{proof}
{ It suffices to show
%\begin{align}
$\cL v_{b^\ast}(x) - (q+\eta)v_{b^\ast}(x) + h(x)=0$ %\label{21}
%\end{align}
with
\begin{align}
h(x) := f(x) + \eta\cM v_{b^\ast}(x)  = f(x)+ {\eta} C (b^\ast -x)^+  + {\eta} v_{b^\ast} (x\lor b^\ast), 
\end{align}
where the second equality holds by Proposition \ref{prop_M}.}
%\blue{[moved here] 
{Because $v_{b^*}$ is smooth enough to apply Ito's formula {(cf. Proposition \ref{Lem404} and Lemma \ref{Lem502})}, it is enough to show that}
%Considering the fact above, 
the process $\{M_t: t\geq 0 \}$, where
\begin{align}
M_t := e^{-(q+\eta)t}v_{b^\ast}(X_t)+\int_0^t e^{-(q+\eta)s} h(X_s) \diff s,
\end{align}
is a {local} martingale with {respect to} the natural filtration $\{\cF^X_t : t\geq 0\}$ generated {by} $X$. {See the proof of \cite[(12)]{BifKyp2010}.}

By the strong Markov property {and because $U^{b^*}_t = X_t$ for $t < T(1)$}, we have, for $x \in\bR$, 
\begin{align*}
v_{b^\ast}(x)  &= \bE_x\left[\int_0^{T(1)} e^{-qt} f({X_t}) \diff t \right]
+C\bE_x \left[  e^{-q T(1)}   ({b^\ast- X_{T(1)}}  )^+  \right]
\\ 
&\qquad%\qquad \qquad\qquad 
 +\bE_x \left[e^{-qT(1)} v_{b^\ast}({ X_{T(1)}\vee b^*} ) \right]\\
&=\bE_x\left[\int_0^{\infty} e^{-(q+\eta)t} f(X_t) \diff t \right]
+ \eta C\bE_x \left[  \int_0^\infty e^{-(q+\eta) t}  (b^\ast- X_{t}  )^+ \diff t \right] \\
&\qquad%\qquad \qquad\qquad 
+ \eta \bE_x \left[\int_0^\infty e^{-(q+\eta)t} v_{b^\ast}(X_t \lor b^\ast  ) \diff t \right]\\
&=\bE_x\left[\int_0^{\infty} e^{-(q+\eta)t} h(X_t) \diff t \right]. %\label{29}
\end{align*}
 %\blue{[I deleted $\mathbb{F}^X$ because it is not used.]}
%\blue{Indeed,} we have,
This together with the strong Markov property gives, for $t\geq 0$ and $\tau_{[n]}:= \inf\{t >0: |X_t| > n\}$ with $n\in\N$, %\blue{[changed $T$ to $\tau$]}
\begin{align}
\bE_x&\left[\int_0^{\infty} e^{-(q+\eta) s} h(X_s) \diff s \mid \cF^X_{t\land \tau_{[n]}}  \right]\\
&\qquad=\int_0^{t\land \tau_{[n]}} e^{-(q+\eta)s} h(X_s) \diff s +\bE_x\left[\int_{t\land \tau_{[n]}}^{\infty} e^{-(q+\eta)s} h(X_s) \diff s \mid \cF^X_{t\land \tau_{[n]}}  \right]\\
&\qquad =\int_0^{t\land \tau_{[n]}} e^{-(q+\eta)s} h(X_s) \diff s +e^{-(q+\eta) (t\land \tau_{[n]})}v_{b^\ast} (X_{t\land \tau_{[n]}}) = M_{t\land \tau_{[n]}}. 
\end{align}
{By the tower property of conditional expectations, $M$ is a local martingale.}
%{[Should we give the proof of $\bE_x\left[  |M_t|\right]<\infty$]}
%where in the last equality, we used \eqref{29}.
%%\blue{[maybe we need to write $\mathcal{F}_t^X$ so that this filtration is only generated by $X$. The filtration for the strategy is generated by $(X,N)$.] }
%By the same reasoning as that of the proof of \cite[(12)]{BifKyp2010}, we obtain \eqref{21}. 
\end{proof}

%By Proposition \ref{prop_M}, $h(x) = f(x)+ \blue{\eta} \cM v_{b^\ast}(x).$ Substituting this in the expression in  Lemma \ref{Lem503} shows the following.
%
%\begin{lemma}\label{Lem504}
%For $x \in\bR$, we have
%$(\cL -q)v_{b^\ast}(x)+\blue{\eta} (\cM v_{b^\ast}(x) -v_{b^\ast} (x)) + f(x)=0, \quad x\in\bR$.
%\end{lemma}
%\begin{proof}
%We compute $\cM v_{b^\ast}$. 
%From \blue{Proposition} \ref{Lem405}, the {derivative $v_{b^\ast}'$} is non-decreasing. In addition, by the definition of $b^\ast$, we have $v_{b^\ast}^\prime(b^\ast)=-C$. 
%Thus, we have 
%\begin{align}
%v_{b^\ast}^\prime (x)
%\begin{cases}
%\leq -C, \quad &x< b^\ast, \\
%\geq -C, \quad &x\geq b^\ast.
%\end{cases}
%\end{align}
%Since the derivative of the function $l \mapsto C l + v_{b^\ast} (x+l)$ is equal to $l\mapsto C+v_{b^\ast}^\prime (x+l)$, {it is minimized when $l = (b^* - x)^+$. Hence,}
%\begin{align}
%\cM v_{b^\ast}(x)= {C (b^\ast - x )^+  + v_{b^*}(x \vee b^*),} \label{23}
%%\begin{cases}
%%C(b^\ast - x )+ v_{b^\ast} (b^\ast)\quad &x < b^\ast, \\
%%v_{b^\ast}(x)\quad & x\geq b^\ast.
%%\end{cases}
%\end{align}
%{and thus $h(x) = f(x)+ \blue{\eta} \cM v_{b^\ast}(x).$ Substituting this in the expression in  Lemma \ref{Lem503} completes the proof.
%}
%%By Lemma \ref{Lem503} and \eqref{23}, we obtain \eqref{22}. 
%\end{proof}

\begin{proof}[{Proof of Theorem \ref{Thm501}}]
%\blue{First suppose Condition \ref{Ass501} holds.}
By %Lemma  \ref{Lem404}, 
{
{Lemmas \ref{Lem301}(iii), \ref{Lem502} and  \ref{Lem503}} and Proposition \ref{Lem404},} the function $v_{b^\ast}$ satisfies the conditions in Lemma \ref{Lem501}. %{[Should we write that $v_b$ is polynomially growth?]}. \blue{[Yes, i think so. Is it easy to show this?]}
%{
%By \eqref{31}, we have 
%\begin{align}
%|v_{b^\ast}(x)|\leq \E_x \left[ \int_0^\infty e^{-qt} (|f(X_t)| + |f(U^{b, \infty}_t)|) \diff t\right]
%+|C| \bE_x \left[ \int_{[0, \infty)}e^{-qt} \diff R^{b,\infty}_t\right]. \label{34}
%\end{align}
%The right hand side of \eqref{34} is polynomial growth by Remark \ref{Rem201} and the proof of \cite[Lemma 3]{NobYam2020} and so is $v_{b^\ast}$.  
%}
Thus,  $v_{b^*}(x) \leq v(x)$ for all $x \in \R$. Because $\pi^{b^*}$ is admissible as in {Proposition \ref{Lem301_2},} the reverse inequality also holds.  This completes the proof {for the case Condition \ref{Ass501} holds.}

{For the case Condition \ref{Ass501} is violated, we {can} write the cost function $f$ in terms of the limit of a sequence of $C^2(\R)$ functions for which Condition \ref{Ass501} is fulfilled and the optimality of a barrier strategy holds. We omit the details because the proof is exactly the same as those proofs given in Section 4.2 in \cite{NobYam2020}.
}
\end{proof}

%\blue{
%\begin{remark}   where the cost function $f$ and its corresponding value function for the case Condition \ref{Ass501} is violated is approximated by those satisfying Condition \ref{Ass501}.  
%\end{remark}
%}

%\subsection{Extension}\label{Sec502}
%Theorem \ref{Thm501} is correct without Assumptions \ref{Ass202}(1) and \ref{Ass501}. 
%To relax Assumptions \ref{Ass202}(1) and \ref{Ass501}, we need give the same approximation as \cite[Section 4.2]{NobYam2020} and \cite[Section 5]{NobYam2020}, respectively. 
%The approximations are the same, thus we omit the proof.  \blue{[Let me work on this part later to make it pursuasive.]}

\section{Convergence as $\eta \to \infty$}  \label{section_convergence}

{In this section, we verify the convergence of the optimal solutions to the classical case \cite{NobYam2020} as the rate of observation $\eta \to \infty$. }
%We do not rely on any result of \cite{NobYam2020} and hence we provide an alternative proof of \cite{NobYam2020} as a byproduct of the obtained results of the paper. 
%We in particular show the convergence of the optimal barriers as well as the value function. }
%We thus  provide a simple alternative proof of the optimality of a barrier strategy of \cite{NobYam2020}.}

%
%{We denote by $R^{b, \eta}$ the control process and by $U^{b, \eta}$ the resulting process when using a periodic barrier strategy at the barrier $b$ in the case where the observation time is defined using a Poisson process with intensity $\eta$. } \blue{[I guess we do not need this sentence because, later we say``Solely in this section''?]}
Recall, in the classical case, strategy $R^\pi$ is any adapted (with respect to the natural filtration of $X$) and nondecreasing process, which does not have to be of the form \eqref{A001}. The classical barrier strategy with barrier $b \in \R$ is given by $R_t^{b,\infty} = (b-\underline{X}_t)^+$ where $\underline{X}$ is the running infimum process of $X$ and the corresponding controlled process is $U_t^{b,\infty} = X_t + R_t^{b,\infty}$, {$t \geq 0$}. As obtained in  \cite{NobYam2020}, {the barrier strategy} $\{R_t^{b^*_\infty,\infty}: t \geq 0\}$ with barrier
\begin{align*}
b^\ast_\infty := \inf\{b\in\bR:\rho_\infty (b) +C \geq 0\} \quad \textrm{for} \quad \rho_\infty(b) := \bE_b \left[\int_0^\infty e^{-qt} f^\prime_+ (U^{b,\infty}_t) \diff t \right], \quad b \in \R,
\end{align*}
{is optimal; the value function becomes}
%the optimal strategy is
%The optimal value function in the classical case is 
$v_\infty^* (x):=\E_x \left[ \int_0^\infty e^{-qt}f(U_t^{b^\ast_\infty,\infty})\diff t + C\int_{0}^\infty e^{-qt} \diff R_t^{b^\ast_\infty,\infty}\right]$ for $x \in \R$.  
%\blue{Note that, in \cite{NobYam2020}, a slightly different barrier $\tilde{b}^\ast_\infty:=\inf\{ b \in \bR:\rho_\infty (b) +C \geq 0\}$ is used, but both selections achieve the same optimal value function. Indeed, if $\tilde{b}^\ast_\infty \neq b^\ast_\infty$ (which necessarily mean $\tilde{b}^\ast_\infty < b^\ast_\infty$), then by \cite[Lemma 6]{NobYam2020}, $b \mapsto v^\infty_b(x):=\E_x \left[ \int_0^\infty e^{-qt}f(U_t^{b,\infty})\diff t + C\int_{0}^\infty e^{-qt} \diff R_t^{b,\infty}\right]$ is continuous and its slope is zero on $(\tilde{b}^\ast_\infty, b^\ast_\infty) \backslash \{x\}$. Therefore $v^\infty_{\tilde{b}^\ast_\infty}(x)=v^\infty_{b^\ast_\infty}(x) = v_\infty^* (x)$. Note in view of \cite[Lemma 5]{NobYam2020} that $b^\ast_\infty$ most likely coincide with $\tilde{b}^\ast_\infty$.
%}
%\blue{[I changed as suggested, but we may be able to stick to our original $b^*_\infty$; see below.]}

%{In fact, in \cite[Theorem 2]{NobYam2020}, we prove that the value $ \tilde{b}^\ast_\infty:=\inf\{ b \in \bR:\rho_\infty (b) +C \geq 0\}$ is an optimal barrier. 
%However, assume that $\tilde{b}^\ast_\infty<b^\ast_\infty$. 
%Then, for $b\in (\tilde{b}^\ast_\infty, b^\ast_\infty)$, the right derivative of the function $v^\infty_b(x):=\E_x \left[ \int_0^\infty e^{-qt}f(U_t^{b,\infty})\diff t + C\int_{0}^\infty e^{-qt} \diff R_t^{b,\infty}\right]$ is equal to $0$ for all $x\in\bR$ by \cite[Lemma 6]{NobYam2020}, and thus $v^\infty_{\tilde{b}^\ast_\infty}(x)=v^\infty_{b^\ast_\infty}(x)$ by the same argument as the proof of \cite[Theorem 1]{NobYam2020}. 
%}

Solely in this section, to spell out the dependence {on the rate $\eta$}, we add super/subscript $\eta$ in an obvious way and add $\infty$ for the classical case.

%\blue{[need to explain more about classical case: TODO]} 
The objective is to show the convergence $b^\ast_{\eta} \to b^*_\infty$ and $v^{*}_{\eta} \to v^{*}_\infty$ where $b^\ast_{\eta}$ is as defined in \eqref{33}. The results hold except for a very particular case when $X$ is the negative of a subordinator (where the reflected process becomes a constant).

\begin{theorem} \label{theorem_conv}
%\blue{[I made this. Could you check?]}
We have (i) $b^\ast_{\eta} \searrow b^*_\infty$ as $\eta \to \infty$ and (ii) $v^{*}_{\eta} \searrow v^{*}_\infty$   as $\eta \to \infty$ {uniformly in $x$ on any compact set,} {where we assume $f'$ is strictly increasing at $b^*_\infty$ for the case $X$ is the negative of a subordinator.}
\end{theorem}
\begin{proof}
Let $\{\eta_n: {n \in \{0\} \cup \mathbb{N}} \}$ be a strictly increasing {(deterministic)} sequence such that $\eta_0 = 0$ and $\eta_n \xrightarrow{n \uparrow \infty} \infty$. Consider, {for each $n$}, a Poisson process $M^{n}$ with rate $\lambda_n := \eta_n - \eta_{n-1} > 0$ independent of $X$, and let
\begin{align}
N^{\eta_n}_t = \sum_{k=1}^n M^{k}_t,\qquad t\geq 0.
\end{align}
%\blue{[change $N^{{\eta_n}}$ to$N^{\eta_n}$?]} 
{We assume $\{M^n: n \geq 1\}$ are mutually independent and also independent of $X$. Hence, {their superposition} $N^{\eta_n}$ becomes}
a Poisson process with rate $\eta_n$ {independent of $X$}. We consider the problems driven by these processes (defined on the same probability space) to show the convergence.

(i) Fix $u \geq 0$. Let $\bar{\sigma}_n(u) := \inf \{ s > u: \Delta N^{\eta_n}_s \neq 0 \}$ and $\underline{\sigma}_n(u) := \sup \{ s < u: \Delta N^{\eta_n}_s \neq 0 \}$  be, respectively, the first {arrival} time after $u$ and the last {arrival} time  before $u$ of $N^{\eta_n}$ (with the understanding $\sup \varnothing = 0$). Then, %because $N^{\eta_n}$ and $X$ are independent, \blue{[I guess it has nothing to do with the independence with $X$? Delete ``because ...''?]}
\[
\p (\bar{\sigma}_n(u) - u > \varepsilon) = \p (N^{\eta_n}_{u + \varepsilon} - N^{\eta_n}_u = 0) = e^{-\varepsilon \eta_n} \xrightarrow{n \uparrow \infty} 0, \quad \varepsilon > 0.
\]
In other words $\bar{\sigma}_n(u) \xrightarrow{n \uparrow \infty} u$ in probability. Because it is decreasing, the convergence also holds in the a.s.-sense. Similarly, we also have $\underline{\sigma}_n(u) \nearrow u$ a.s.\ {as $n \to \infty$.}

%\blue{[People use $G$ instead of $L$?]} 
Fix $t > 0$ and $G(t):=\sup\{s\in[0, t]:X_{s-}\wedge X_s= {\underline{X}_t}  \}$ {(with $X_{0-}=X_0$)}. Suppose $G(t) {\in (0, t)}$.  If $X_{G(t)-} \geq X_{G(t)}$ (i.e.\ $X$ is continuous or {jumps downward} at $G(t)$), because $X$ is right-continuous a.s.,
\[
X_{\bar{\sigma}^n (G(t))} \xrightarrow{n \uparrow \infty} X_{G(t)} = X_{G(t)} \wedge X_{G(t)-} = \underline{X}_t.
\]
If $X_{G(t)} > X_{G(t)-}$ (i.e.\ $X$ {jumps upward} at $G(t)$), then
\[
X_{\underline{\sigma}^n (G(t))} \xrightarrow{n \uparrow \infty} X_{G(t)-} = X_{G(t)} \wedge X_{G(t)-} =  \underline{X}_t.
\]
%Alternatively, we can write
%\[
%R_t^b = \max_{1 \leq k \leq N_t} (b-X_{T(k)}) \vee 0, \quad t \geq 0,
%\]
{These together with Remark \ref{remark_simple_expression_R} give,} for any $b \in \R$, %\blue{[below changed $\vee 0$ to $^+$. Please check.]}
\begin{align*}
R^{b, \eta_n}_t  { = \max_{1 \leq k \leq N_t^{\eta_n}} (b-X_{T(k)})^+}
%= \sup_{s: \Delta N_s^n \neq 0}1_{\{s \leq t \}} \left( b - X_t\right) \vee 0 
\geq (b-X_{\underline{\sigma}^n (G(t))})^+ \vee (b-X_{\bar{\sigma}^n (G(t))})^+ \xrightarrow{n \uparrow \infty} (b - \underline{X}_t)^+ {= R_t^{b,\infty}}.
\end{align*}
%{[I guess 
%\begin{align*}
%R^{b, \eta_n}_t  { = \max_{1 \leq k \leq N_t^{\eta_n}} (b-X_{T(k)})^+}
%%= \sup_{s: \Delta N_s^n \neq 0}1_{\{s \leq t \}} \left( b - X_t\right) \vee 0 
%\geq (b-X_{\bar{\sigma}^n (G(t))})^+ 
%\end{align*}
%may not true. 
%]}
For the case $G(t) = 0$ {(i.e., $\underline{X}_t = X_0$)}, by slightly modifying the arguments, $R^{b, \eta_n}_t  
%= \sup_{s: \Delta N_s^n \neq 0}1_{\{s \leq t \}} \left( b - X_t\right) \vee 0 
\geq  (b-X_{\bar{\sigma}^n (0)})^+ \xrightarrow{n \uparrow \infty} (b - X_0)^+ {= R_t^{b,\infty}}$.

{If $G(t) = t$, we have $X_{\underline{\sigma}^n (t)} \xrightarrow{n \uparrow \infty} X_{t-}$ and hence
$R^{b, \eta_n}_t  \xrightarrow{n \uparrow \infty} (b - \underline{X}_{t-})^+$, 
which differs from $R_t^{b,\infty}$ only when $X$ jumps downward at $t$.
}

{By these and because $n \mapsto R^{b,\eta_n}_t$ is increasing, we have}
 %and bounded from above by $R^{b,\infty}_t$. Hence, 
 $R^{b,\eta_n}_t \nearrow  R^{b,\infty}_t$ and consequently $U^{b,\eta_n}_t \nearrow U^{b,\infty}_t$ as $n \to \infty$ for {a.e.} $t > 0$ {(more specifically all $t > 0$ except $t$ at which $X$ jumps downward)}
 %$R^{b, \infty}$ does not have a jump}, 
for all $ b \in \R$. 
% {[I guess $R^{b,\eta_n}_t \nearrow  R^{b,\infty}_t$ may not be true when $R^{b,\infty}$ has jump at time $t$.]
%% For any $\varepsilon>0$, we have 
%%\begin{align*}
%%R^{b, \eta_n}_{t+\varepsilon}  { = \max_{1 \leq k \leq N_{t+\varepsilon}^{\eta_n}} (b-X_{T(k)})^+}
%%\geq \lim_{n\uparrow\infty }(b-X_{\underline{\sigma}^n (G(t))})^+ \vee (b-X_{\bar{\sigma}^n (G(t))})^+ = (b - \underline{X}_t)^+ {= R_t^{b,\infty}}.
%%\end{align*}
% }

%\blue{[Actually, $U_t^{b, \eta_n}$ may coincide with $U_t^{b, \infty}$ and so it can happen that  $f'_+ (U_t^{b, \eta_n}) \xrightarrow{n \uparrow \infty} f'_- (U_t^{b, \infty})$ or  $f'_+ (U_t^{b, \eta_n}) \xrightarrow{n \uparrow \infty} f'_+ (U_t^{b, \infty})$. In any case, 
%$f'_- (U_t^{b, \infty}) \leq \lim_{n \to \infty}f'_+ (U_t^{b, \eta_n})  \leq f'_+ (U_t^{b, \infty})$ and so $\rho_\infty(b-) \leq \lim_{n \to \infty} \rho_{\eta_n} (b)  \leq \rho_\infty(b)$. So I changed as follows]}

By this, together with the fact that {$f_+'$} is 
%continuous a.e.\ {[The continuity of $f^\prime_+$ is not necesarry?]} and
 nondecreasing, we have $f'_- (U_t^{b, \infty}) \leq \lim_{n \to \infty}f'_+ (U_t^{b, \eta_n})  \leq f'_+ (U_t^{b, \infty})$ for a.e.\ $t > 0$ and hence, by the monotone convergence theorem,
%{we have 
%\begin{align}
%\lim_{n\uparrow\infty}\rho_{\eta_n}(b)\leq \rho_{\infty}(b) \leq 
%\lim_{n\uparrow\infty}\bE_b \left[\int_0^\infty e^{-qt} f^\prime_+ (X_t+R^{b, \eta_n}_{t+\varepsilon}  ) \diff \right]
%\end{align}
%}
$n \mapsto \rho_{\eta_n}(b)$ is nondecreasing and  $\rho_{\infty}(b-) \leq \lim_{n \to \infty} \rho_{\eta_n}(b) \leq \rho_{\infty}(b)$ for all $b \in \R$.

This shows that $b^\dagger := \lim_{n \to \infty} b^*_{\eta_n}$ exists and $b^\dagger \geq b^*_\infty$. The monotonicy  also suggests that $\rho_{\eta_n} (b^\dagger) \leq {-C}$ uniformly in $n$ and hence $\rho_{\infty} (b^\dagger-) \leq {-C}$. {If $b^\dagger > b^*_\infty$, then we must have $\rho_{\infty} ((b^*_\infty + b^\dagger)/2) \leq {-C}$.} {However, as}  shown in  \cite[Lemma 5]{NobYam2020},
%{[Should we change the definition of $b^\ast_\infty$ as above?]},
 $\rho_\infty (b)>-C$ for $b>b^\ast_\infty$ {for the case $X$ is not the negative of a subordinator. For the case it is the negative of a subordinator, we have $\rho_\infty(b) = f'_+(b)/q$, which is strictly increasing at $b = b^*_\infty$ by assumption  and hence the contradiction can be derived similarly.}
 Hence we must have $b^\dagger = b^*_\infty$, as desired. %\blue{[I don't know but where do we use the new version of $b^*$? I feel we can just stick to our original $b^*$.]}
(ii) 
Fix $N \in \N$ and $x\in\bN$. Because, for $n \geq N$,  $b^*_\infty \leq  b^*_{\eta_n} \leq b^*_{\eta_N}$ and hence $U_t^{b^*_\infty,\eta_N} \leq U_t^{b^*_{\eta_n},\eta_n} \leq U_t^{b^*_{\eta_N}, \infty}$ and by the convexity of $f$,
\begin{align}
\bE_x \left[\int_0^\infty e^{-qt} \sup_{n \geq N}|f (U_t^{b^*_{\eta_n},\eta_n})| \diff  t\right] 
\leq  \bE_x \left[\int_0^\infty e^{-qt} (| f (U_t^{b^*_\infty,\eta_N}) | +  | f (U_t^{b^*_{\eta_N}, \infty}) |+c)\diff  t\right] < \infty,
\end{align}
where $c$ is a constant value defined in \eqref{35}. 
{On the other hand, $|U^{b^{\ast}_{\eta_n},\eta_n}_t - U^{b^{\ast}_{\infty},\infty}_t | \leq |U^{b^{\ast}_{\eta_n},\eta_n}_t - U^{b^{\ast}_{\infty},\eta_n}_t| + |U^{b^{\ast}_{\infty},\eta_n}_t  - U^{b^{\ast}_{\infty},\infty}_t| \xrightarrow{n \uparrow \infty} 0$ by Remark  \ref{remark_simple_expression_R} and (i) for a.e.\ $t > 0$.}
Hence dominated convergence gives the pointwise convergence of $\E_x \big[ \int_0^\infty e^{-qt}f(U_t^{b^\ast_{\eta_n},\eta_n})\diff t \big]$ to $\E_x \big[ \int_0^\infty e^{-qt}f(U_t^{b^\ast_\infty,\infty})\diff t \big]$ for all $x \in \R$.

{On the other hand, by integration by parts,}
\begin{align}
&\left| \bE_x \left[
%\int_{[0, \infty)} 
{\int_0^\infty}
e^{-qt}\diff R_t^{b^*_{\eta_n},\eta_n} \right] -  \bE_x \left[\int_{[0, \infty)} e^{-qt}\diff R_t^{b^*_{\infty},\infty} \right] \right| \\
&
\qquad\qquad\qquad\qquad
=q\left| \bE_x \left[
%\int_{[0, \infty)} 
{\int_0^\infty} e^{-qt} R_t^{b^*_{\eta_n},\eta_n} \diff t\right] -  \bE_x \left[
%\int_{[0, \infty)} 
{\int_0^\infty} e^{-qt} R_t^{b^*_{\infty},\infty} \diff t\right] \right|\\
&\qquad\qquad\qquad\qquad
\leq q  \bE_x \left[
%\int_{[0, \infty)} 
{\int_0^\infty}e^{-qt} \left| R_t^{b^*_{\eta_n},\eta_n}- R_t^{b^*_{\infty},\infty}\right| \diff t\right]. 
%\leq  2\bE_x \left[\int_{[0, \infty)} e^{-qt} R_t^{b^*_{\eta_N},\infty} \diff t\right]
%<\infty. 
\end{align}
{Here, $ \bE_x \big[\int_0^\infty e^{-qt} \sup_{n \geq N}| R_t^{b^*_{\eta_n},\eta_n}- R_t^{b^*_{\infty},\infty} | \diff t\big] 
\leq  2\bE_x \big[\int_0^\infty e^{-qt} R_t^{b^*_{\eta_N},\infty} \diff t \big]
<\infty$ thanks to $ R_t^{b^*_{\eta_n},\eta_n}\lor R_t^{b^*_{\infty},\infty} \leq R_t^{b^*_{\eta_N}, \infty}$ for all $t > 0$.}
% and thus, 
%\blue{[I guess we need to be careful at time zero especially when $x$ is below $b$?]}
Thus, by the dominated convergence theorem together with $|R^{b^{\ast}_{\eta_n},\eta_n}_t - R^{b^{\ast}_{\infty},\infty}_t |=|U^{b^{\ast}_{\eta_n},\eta_n}_t - U^{b^{\ast}_{\infty},\infty}_t | \xrightarrow{n \uparrow \infty} 0$ for a.e. $t > 0$, we have the pointwise convergence of $\bE_x \big[\int_0^\infty e^{-qt}\diff R_t^{b^*_{\eta_n},\eta_n} \big]$ to $\bE_x \big[\int_{[0, \infty)} e^{-qt}\diff R_t^{b^*_{\infty},\infty} \big]$ for all $x \in \R$.

Finally, because $n \mapsto v^{\ast}_{\eta_n}(x)$ is monotone and each value function is continuous in $x$, $\lim_{n \uparrow \infty}v^{\ast}_{\eta_n}(x) = v^{\ast}_\infty(x)$ holds uniformly in $x$ on any compact set by Dini's theorem.

 \end{proof}
%Hence
%\[
%\tilde{v}(x)-C = \E_x [\int_0^{\tau_b^-} e^{-qt} f(X_t) \diff t] + \E [e^{-q \tau_b} (b - X_{\tau_b^-} + (\tilde{v}(b)-C))]
%\]
%Solving for $C$,
%\[
%\tilde{v}(x) - \E_x [\int_0^{\tau_b^-} e^{-qt} f(X_t) \diff t] - \E_x [e^{-q \tau_b} (b - X_{\tau_b^-} + \tilde{v}(b))] = C (1 - \E_x [e^{-q \tau_b} ])
%\]
%Hence,
%\[
%C = \lim_{x \rightarrow \infty} (\tilde{v}(x) - \E_x [\int_0^{\tau_b^-} e^{-qt} f(X_t) \diff t] ) 
%\]

\section{Numerical results} \label{section_numerics}

%\blue{[added]}
In this section, we confirm the obtained results through numerical experiments via Monte Carlo simulation (classical Euler scheme). 
%While there exist many simulation techniques for \lev processes, it is not our focus to evaluate and compare the performance of these techniques.  Here, in order to confirm that our results are easily implementable, we use a classical Euler scheme focusing on the case $X$ is the sum of drifted Brownian motion and a compound Poisson process with two-sided jumps (i.e.\ $\Pi$ is a finite measure). 
%At least in theory, any \lev process can be approximated using those with  finite \lev measures.  
%An obvious shortcoming of the Euler scheme is that it fails to be accurate  when $f'$ is not continuous (for the computation of $b^*$), and hence here we focus on the cases $f'$ is continuous. This issue may potentially be resolved by using other methods, such as the Wiener-Hopf simulation \cite{KKPV}	 which is implementable if the Wiener-Hopf factorization is explicitly known. 
%As an example, we consider the following \lev process of the form:TODO.
In order to confirm that the results hold for a wide class of \lev processes, we choose a \lev process
%.  We consider 
$X$ of the form
\begin{equation}
 X_t = X_0 -0.1 t+ 0.2 B_t + \sum_{n=1}^{N_t^+} Z_n^{+} - \sum_{n=1}^{N_t^-} Z_n^{-}, \quad 0\le t <\infty, \label{X_phase_type}
\end{equation}
where $\{B_t: t\ge 0\}$ is a standard Brownian motion and $\{N_t^+: t\ge 0 \}$ and $\{N_t^-: t\ge 0 \}$ are Poisson processes with arrival rates $0.4$ and $0.6$, respectively. The upward and downward jumps  $\{ Z_n^+: n \in \mathbb{N} \}$ and $\{ Z_n^-: n \in \mathbb{N} \}$ are i.i.d.\ sequences of (folded) normal random variables with mean zero and variance $1$ and Weibull random variables with shape parameter $2$ and scale parameter $1$, respectively. These processes are assumed mutually independent.

For the running cost function $f$, we consider the following three cases: 
\begin{align} \label{examples_f}
\begin{aligned}
&\ \ f_1(x) := x^2, \; f_2(x) :=x^3 1_{\{x \geq 0\}} +  x^2 1_{\{x < 0\}}, \\
& f_3(x) := [ x^2 + e^{-(x-1)} ] 1_{\{x \geq 1\}} + \frac {x^2+3} 2 1_{\{x < 1\}},
 \end{aligned}
\end{align}
for $x \in \mathbb{R}$, which are convex and continuously differentiable on $\mathbb{R}$.   For other parameters, we set $q=0.05$ and $C = 1$. For each realization, we truncate the time horizon to $T = 100$ and discretize $[0,T]$ using $N = 10,000$ equally-spaced points with distance $\Delta_t := T/N$. Unless stated otherwise, we use $\eta = 1$.
%Here, we consider $f_2$ and $f_3$ as examples where they fail to be $C^2$ at $0$ and $1$, respectively (i.e.\ the condition in Lemma \ref{Lem205a} does not hold), and hence the $C^2$ property of $v_{b^*}$ as in Assumption \ref{assump_C_line} needs to be verified numerically.

%As shown in the previous section, obtaining the optimal solution reduces to computing the barrier  $b^*$ as in \eqref{def_b_star}. Note that  \eqref{def_b_star} is equivalently written as:
%\begin{align}
%b^\ast {:=}\inf\left\{b\in\bR: \blue{\rho(b) := }\bE_0 \left[\int_0^\infty e^{-qt} f^\prime_+(U^{0}_t+b)\diff t\right] +C \geq 0\right\}.  
%\end{align}
For the approximation of the expectation, we first obtain a set of $M := 5,000$ sample paths of $X$ started at zero, say $\hat{\mathbf{X}} := (\hat{X}^{(1)}, \ldots, \hat{X}^{(M)})$ with $\hat{X}^{(m)} = \{\hat{X}^{(m)}_{n \Delta_t}: 1 \leq n \leq N \}$ for $1 \leq m \leq M$. Control opportunities $\hat{\mathbf{N}}^\eta := (\hat{N}^{\eta,(1)}, \ldots, \hat{N}^{\eta,(M)})$ with $\hat{N}^{\eta,(m)} = \{\hat{N}^{\eta,(m)}_{n \Delta_t  }: 1 \leq n \leq N \}$ for $1 \leq m \leq M$ are sampled by generating  $\hat{N}^{\eta,(m)}_{(n+1) \Delta_t} - \hat{N}^{\eta,(m)}_{n \Delta_t} = 1_{\{\mathrm{e} < \Delta_t \}}$ with i.i.d.\ $\mathrm{e} \sim \exp(\eta)$ and  their corresponding reflected paths (with barrier zero) $\hat{U}^{0,(m)} = \{\hat{U}^{0,(m)}_{n \Delta_t }: 1 \leq n \leq N \}$ are then computed.
% and their corresponding reflected paths (with barrier zero) \blue{TODO} $\hat{U}^{0,(m)}_{n \Delta_t } = \hat{X}_{n \Delta_t}^{(m)} - \min_{0 \leq l \leq n} \hat{X}_{l \Delta_t}^{(m)}$, $m = 1, \ldots, M$. 
 These sample paths can be used commonly for the approximation of the expectation in $\rho(b)$ as in \eqref{def_rho}. In other words, we approximate it by $\hat{\rho}_M(b) :=M^{-1} \sum_{m=1}^M \Delta_t \sum_{n=0}^{N} e^{-q n \Delta_t } f'(\hat{U}^{0,(m)}_{n \Delta_t   }+b)$.
As shown in Section \ref{section_barrier}, $\rho(b)$ is monotone and hence $b^*$ can be obtained by classical  bisection. While $\hat{\rho}_M(b)$ for each $b$ is an approximated value, because we are using the same sample paths $(\hat{\mathbf{X}}, \hat{\mathbf{N}}^\eta)$,  the monotonicity of $b \to \hat{\rho}_M(b) $ is still preserved, causing no problem in using bisection methods.  %Below, we use $M = TODO$.
Figure \ref{plot_b_star} shows the plots of $\hat{\rho}_M(b)$ for Cases $i$ for $i = 1,2,3$. It can be confirmed that it is indeed monotonically increasing, and the root becomes $b^*$. Note for the case $i=1$, $\rho_M(b)$ becomes a  straight line. % as discussed in Example \ref{example_periodic}. \blue{TODO}

\begin{figure}[htbp]
\begin{center}
\begin{minipage}{1.0\textwidth}
\centering
\begin{tabular}{cc}
 \includegraphics[scale=0.6]{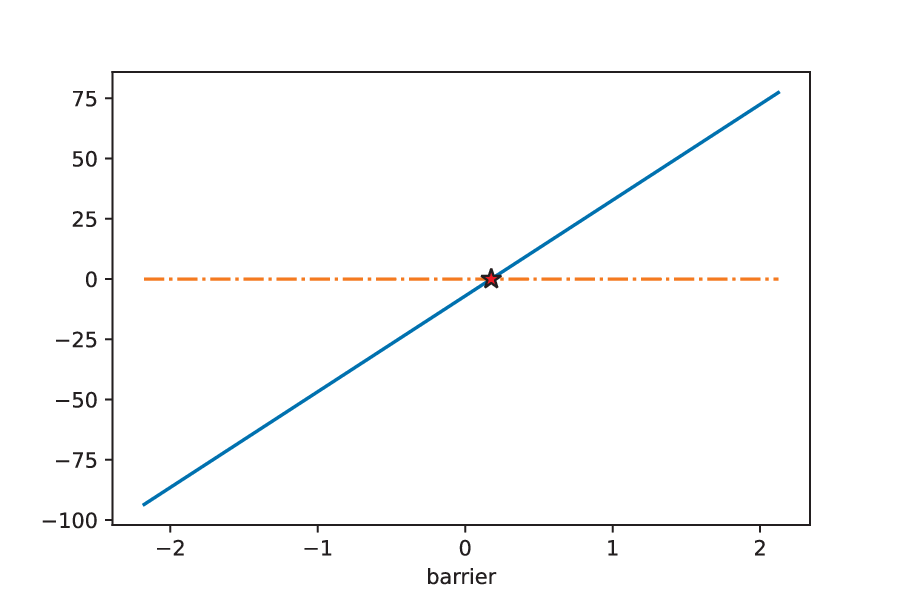} \\ Case 1 \\ \includegraphics[scale=0.6]{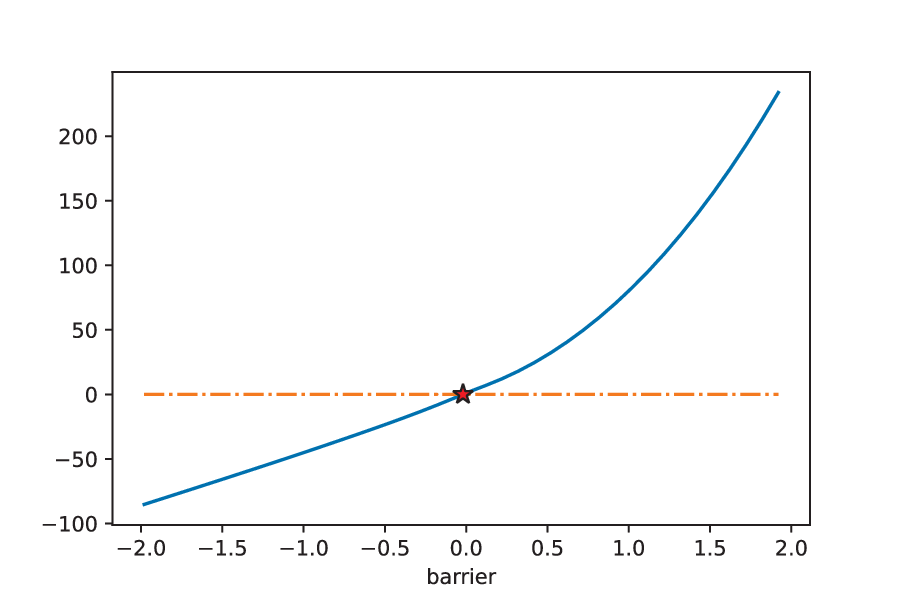}   \\
  Case 2 
 \end{tabular}
  \includegraphics[scale=0.6]{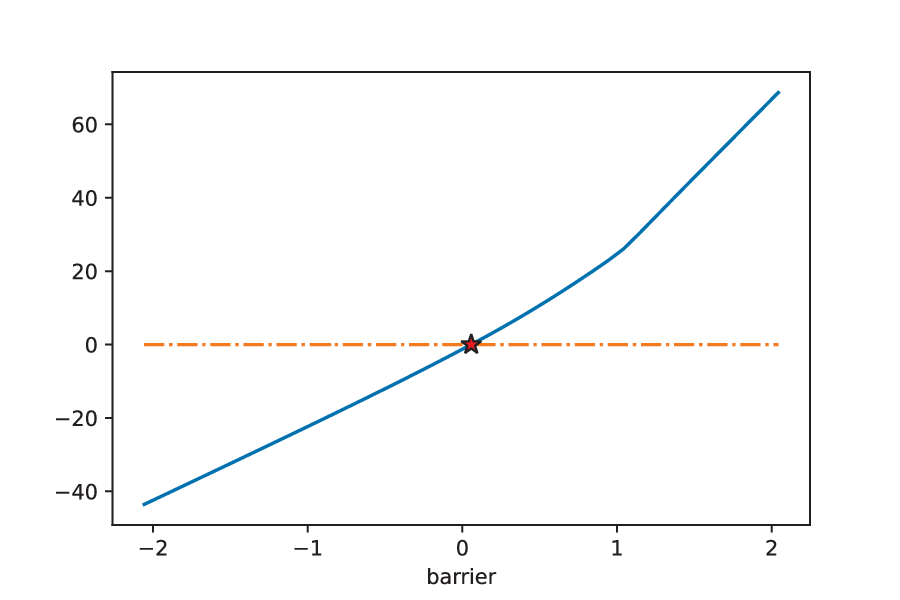} \\ Case 3
\end{minipage}
\caption{Plot of $\hat{\rho}_M(b)$ for Case $i$ under the cost function $f_i$ as in \eqref{examples_f}, for $i = 1,2,3$. The root (indicated by a star) becomes an approximation of the optimal barrier $b^*$.
} \label{plot_b_star}
\end{center}
\end{figure}

With the approximated optimal barrier $b^*$, we shall now confirm the optimality by comparing the expected total costs $v_{b^*}$ with $v_b$ under suboptimal choices of $b$. In order to compute these, we continue using the set of paths  $(\hat{\mathbf{X}}, \hat{\mathbf{N}}^\eta)$. 
%The reflected path with lower barrier $b$ under $\p_x$ becomes
%$\hat{U}^{b,(m)}_{n\Delta_t } = (\hat{X}_{n\Delta_t }^{(m)} +x ) - \min_{0 \leq l \leq n} \{ (\hat{X}_{l \Delta_t }^{(m)}+x) - b) \wedge 0 \}$, $m = 1, \ldots, M$. 
Figure \ref{figure_value_function} shows the results. It can be confirmed that the selection $b^*$ indeed minimizes the total expected cost for  all starting points. 
%approximated values of the optimal value function along with those under suboptimal selections of $b$.
 
Finally, we confirm the convergence as $\eta \to \infty$. In Figure \ref{figure_derivatives}, we plot the value function when $\eta = 2,5,10,20,50,100,200,500,1000$ together with the classical case whose reflected path with lower barrier $b$ under $\p_x$ is approximated by
$(\hat{X}_{n\Delta_t }^{(m)} +x ) +  \max_{0 \leq l \leq n} (b- (\hat{X}_{l \Delta_t }^{(m)}+x) )^+$. It is observed in all cases that the optimal barrier and the value function converge decreasingly to those of the classical case, confirming  Theorem \ref{theorem_conv}.

%In view of the assumption that $v_{b^*}$ is twice differentiable for Case 2 and 3, it seems so.

%
\begin{figure}[htbp]
\begin{center}
\begin{minipage}{1.0\textwidth}
\centering
\begin{tabular}{cc}
\includegraphics[scale=0.6]{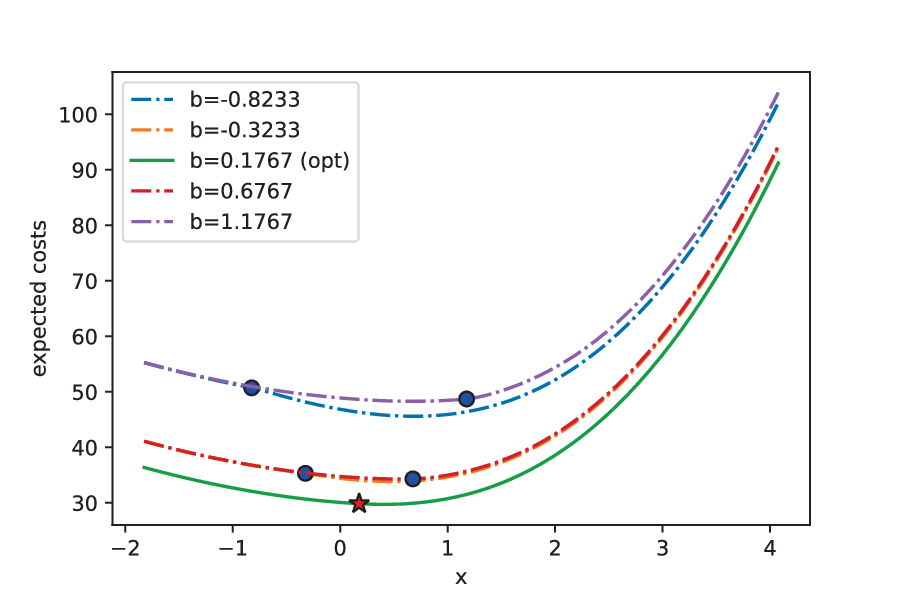} \\ Case 1
\\
 \includegraphics[scale=0.6]{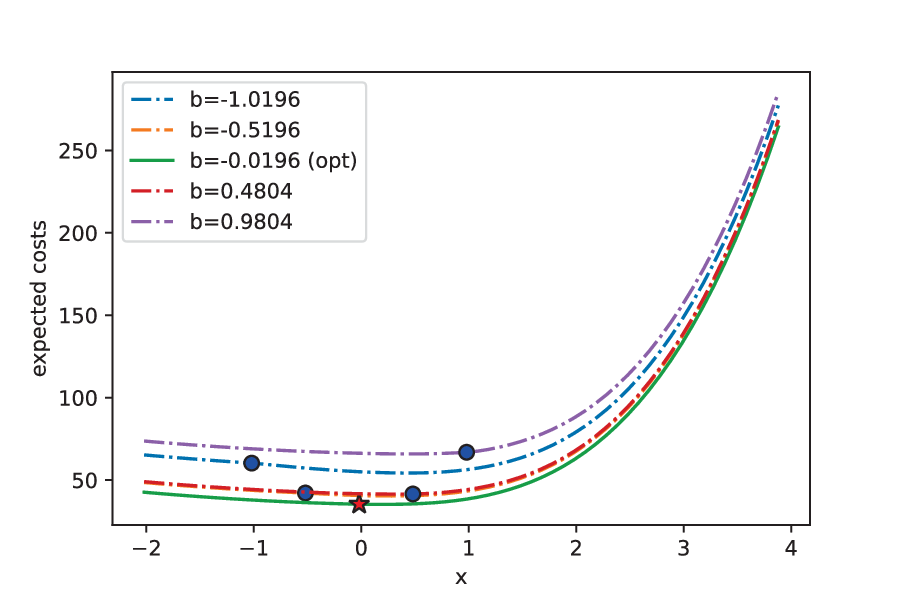}     \\
  Case 2 
 \end{tabular}
  \includegraphics[scale=0.6]{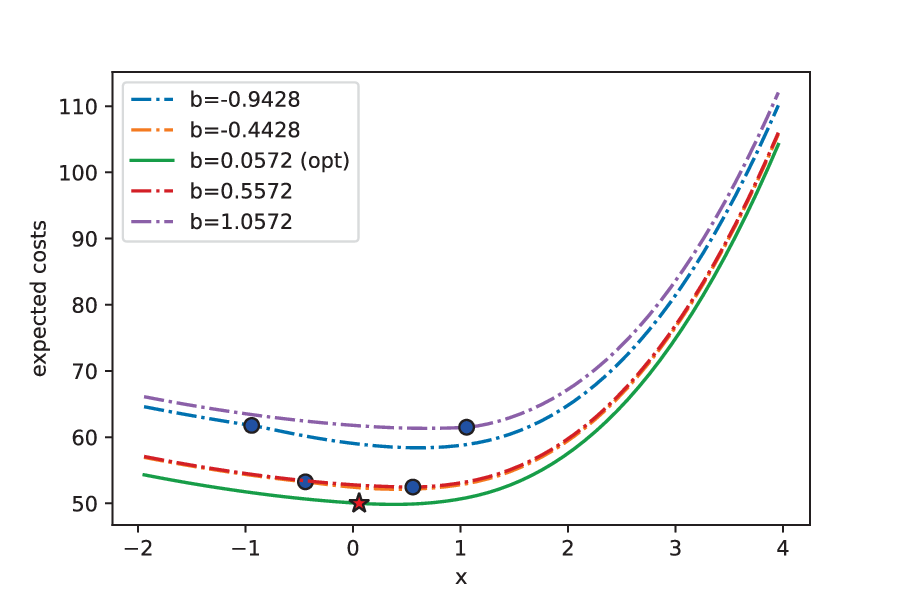} \\
  Case 3
\end{minipage}
\caption{Plot of the approximated value functions $v_{b^*}$ (solid) along with $v_b$ (dotted) for $b = b^*-1, b^*-0.5,  b^*+0.5,b^*+1.0$ for Case $i$ for $i = 1,2,3$. The points at the barriers are indicated by stars and circles for $b = b^*$ and $b \neq b^*$, respectively.
} \label{figure_value_function}
\end{center}
\end{figure}

%The reflected path with lower barrier $b$ under $\p_x$ becomes

%smoothness of the value function $v_{b^*}$.
%We compute the derivative $v_{b^*}'$ as in \eqref{40} via simulation using the same sample paths $\hat{\mathbf{X}}$. In order to evaluate the second derivative, we compute $(v_{b^*}'(x+\delta)-v_{b^*}'(x))/\delta$ where $\delta = 0.01$. These results are summarized in Figure \ref{figure_derivatives}. While Cases 2 and 3 do not satisfy the conditions in Lemma \ref{Lem205a}, the approximation of the second derivative appears to be continuous, justifying Assumption  \ref{assump_C_line}. \blue{TODO}
%
%\blue{TODO: about convergence in $\eta$}

\begin{figure}[htbp]
\begin{center}
\begin{minipage}{1.0\textwidth}
\centering
\begin{tabular}{cc}
\includegraphics[scale=0.6]{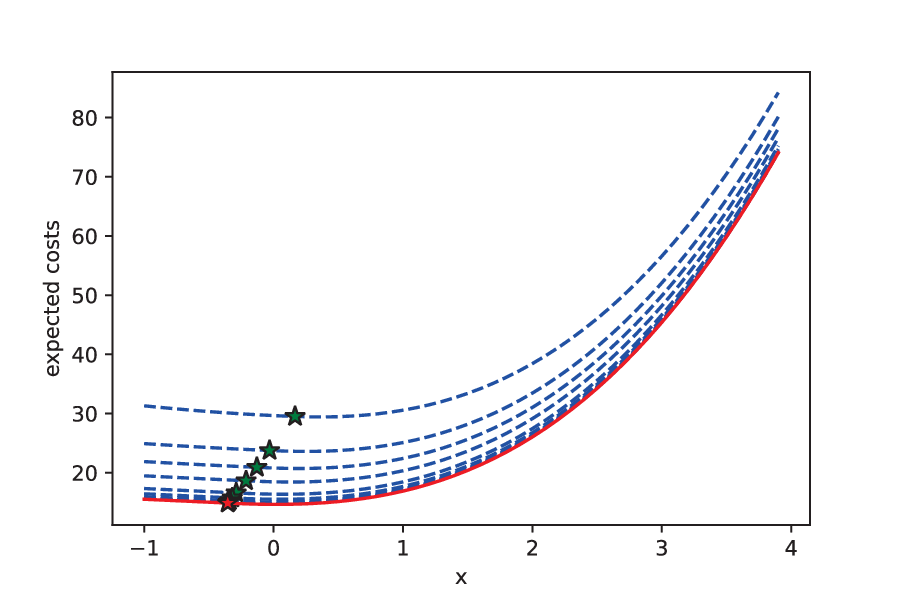} \\ Case 1\\
 \includegraphics[scale=0.6]{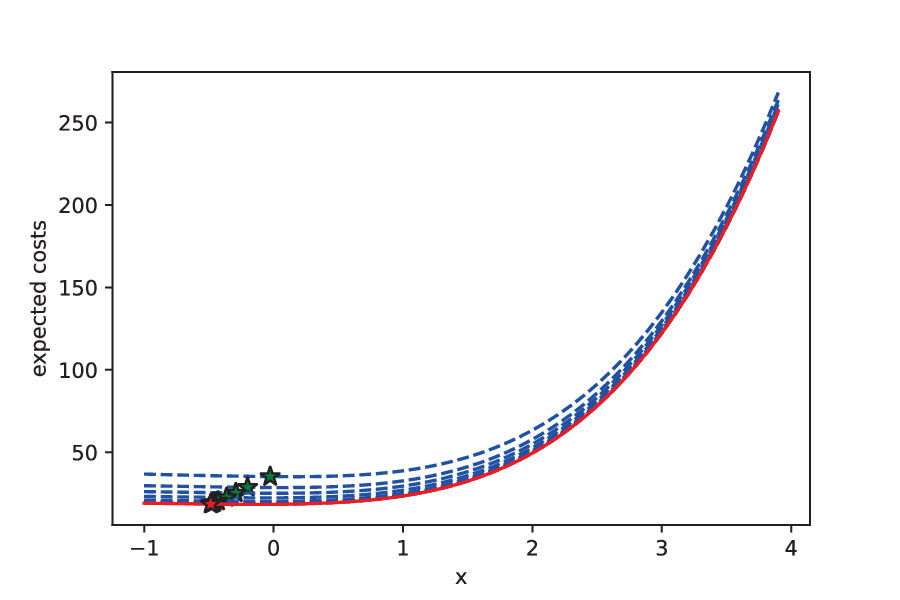}    \\
   Case 2 
 \end{tabular} \\
 \includegraphics[scale=0.6]{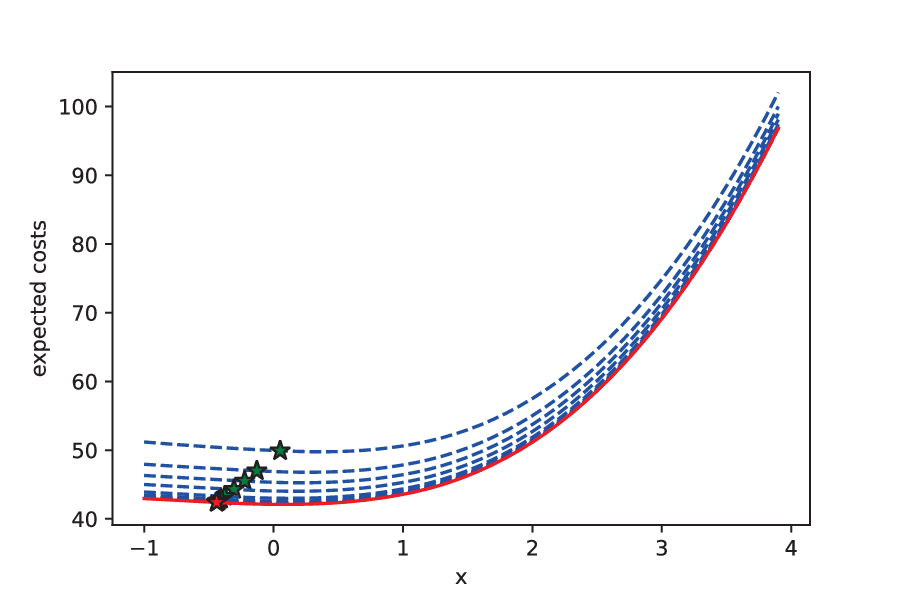} \\
 Case 3
\end{minipage}
\caption{Plot of the approximated value functions $v_{b^*}$ (dotted) for $\eta = 2,5,10,20,50,100,200,500,1000$ along with that in the classical case (solid). The points at the barriers are indicated by stars.
} \label{figure_derivatives}
\end{center}
\end{figure}

\section{Concluding remarks} \label{section_concluding_remarks}

{In this paper, we solved the stochastic control problem of minimizing the sum of running and controlling costs under the constraint that control opportunities are restricted to independent Poisson arrival times. For a general \lev process model, we showed the optimality of a simple barrier strategy, with its barrier  analytically provided as a root of the equality \eqref{opt_b_review}. Furthermore, we demonstrated that the optimal solutions in the Poissonian setting converge to those in the continuous-observation setting. These results potentially provide a new approach to the classical case, {using techniques developed}
%by way of the techniques developed
 in this paper for Poisson observation models.
}

  {One important extension is to consider the case where one can control the process in both directions. This scenario has been studied in the continuous-observation case driven by spectrally negative Lévy processes, as discussed in \cite{Baurdoux_Yamazaki}, where it is shown that it is optimal to reflect the process at both upper and lower barriers.}

 {Another natural extension is to consider the case with a fixed intervention cost. In this case, the optimal strategy is expected to be of the two-barrier type. More specifically, it is expected to be a variant of the $(s,S)$-policy (see, e.g., \cite{Benkerouf_Bensossan, Bensoussan_Liu_Sethi}), which moves the process to a certain point, say $\bar{b}$, whenever it is observed to be below a different point, say $\underline{b}$, at Poisson observation times. This is a reasonable conjecture based on Yamazaki \cite{Yam2017}, who showed the optimality of such a policy for the continuous-observation case under a spectrally one-sided Lévy model.}

%If your paper includes appendices, then precede the first of them by the command
\appendix
%and then carry on using the \section and \subsection commands, as above.

\section{Proofs} \label{appendix_proof}

\subsection{The proof of Lemma \ref{Lem301}}\label{SecA01}
%Among the conditions of the admissibility, the one that is not clear is the finiteness condition \eqref{1}. 
We fix $b, x\in \bR$. {Let $U^{b,\infty}$ and $R^{b,\infty}$ be those defined in Section \ref{section_convergence} for the controlled and control processes in the classical setting under the barrier strategy with barrier $b$. }
%We prove 
%\begin{align}
%\E_x \left[ \int_0^\infty e^{-qt}{|f({U^b_t  })|}\diff t\right]<\infty, \label{2}
%\end{align}
%and
%\begin{align}
%\E_x \left[ \int_{[0, \infty)} e^{-qt}\diff R^b_t\right]<\infty, \label{30}
%\end{align}
%respectively. 
%Maybe it would be cleaner to write like this? 
%\blue{[to be consistent, use $\infty$ instead of tilde.]}
First we have a bound
\begin{align}
X_t \leq U_t^b \leq U_t^{b,\infty},\quad 0 \leq R^b_t \leq R^{b,\infty}_t.
\label{31}
\end{align}
%where $U_t^{b,\infty} = X_t + R^{b,\infty}_t$ is the classical reflected process at $b$ with $R^{b,\infty}_t = \sup_{s\in[0, t]} (b-X_s)^+$. 
%\begin{align}
%\tilde{U}_t^b= X_t + {\sup_{s\in[0, t]} (b-X_s )\lor 0}, \qquad t\geq 0.
%\end{align}}
  By the convexity of $f$, we have
$|f(U_t^b)| \leq |f(U_t^{b,\infty})| + |f(X_t)| + c$
where
\begin{align}\label{35}
c=
\begin{cases}
|\inf_{y\in \R}f(y) |, \quad &\text{ if it exists}, \\
 0, \quad &\text{ otherwise}.
 \end{cases}
\end{align}
By Remark \ref{Rem201} and \cite[(A.3)]{NobYam2020} under Assumptions \ref{Ass201} and \ref{Ass202}, we obtain {(i)}.
%{
%\begin{align}
%\E_x \left[ \int_0^\infty e^{-qt}{|f({U^b_t  })|}\diff t\right]<\infty.
%\label{2} 
%\end{align}}
By \eqref{31} and since \cite[Lemma 3]{NobYam2020} {holds} under Assumptions \ref{Ass201} and \ref{Ass202}, we obtain {(ii).} 

%By \eqref{31} \blue{and using $c$ in Section \ref{SecA01}}, we have 
%\blue{[moved here]} 
We have
\begin{align}
|v_b(x)|\leq \E_x \left[ \int_0^\infty e^{-qt} (|f(X_t)| + |f(U^{b, \infty}_t)| {+c}) \diff t\right]
+|C| \bE_x \left[ \int_{[0, \infty)}e^{-qt} \diff R^{b,\infty}_t\right], \label{34}
\end{align}
%The right hand side of \eqref{34} is
which is of
 polynomial growth by Remark \ref{Rem201} and the proof of \cite[Lemma 3]{NobYam2020}, {showing (iii).}

\subsection{Proof of Lemma \ref{Lem302}} \label{proof_Lem302}
Note that $T_b=T(k)$ for some $k\in\bN$, almost surely on $\{T_b < \infty\}$.
By the monotone convergence theorem and  the strong Markov property, we have 
\begin{align}
&\lim_{\varepsilon\downarrow 0}\bP_x \left(T_b= T_{b+\varepsilon}  , T_b < \infty \right)=\lim_{\varepsilon\downarrow 0}\sum_{k\in\bN}\bP_x \left(T_b= T_{b+\varepsilon}=T(k)  \right)\\
&=\sum_{k\in\bN}\lim_{\varepsilon\downarrow 0}\bP_x \left(X_{T(1)}\geq b+\varepsilon, X_{T(2)}\geq b+\varepsilon, \dots, 
X_{T(k-1)}\geq b+\varepsilon, X_{T(k)}<b \right)\\
&=\sum_{k\in\bN} \bP_x \left(X_{T(1)} >  b, X_{T(2)} >  b, \dots, 
X_{T(k-1)} > b, X_{T(k)}<b \right)\\
&=%\sum_{k\in\bN}\bE_x \left[1_{\{X_{T(1)}> b\}}\bE_{X_{T(1)}}\left[ 1_{\{X_{T(1)}> b\}} \bE_{X_{T(1)}} \left[\cdots 
%1_{\{X_{T(1)}> b\}} \bP_{X_{T(1)}}\left( X_{T(1)} <b \right) \cdots \right]\right]\right].\label{25}
\sum_{k\in\N} E^{(k)}(x), \label{25}
\end{align}
%{[How about changing the last expression above to:]
%\begin{align}
%=\sum_{k\in\N} E^{(k)} \bP_{\cdot}(X_{T(1)}<b) (x), 
%\end{align}
%where for non-negative measurable function $g$, we define 
%\begin{align}
%E^{(1)} g(x) = \bE_x \left[1_{\{X_{T(1)} >b\}} g(X_{T(1)})\right]
%,\quad E^{(l+1)} g(x) = \bE_x \left[1_{\{X_{T(1)} >b\}} E^{(l)}g(X_{T(1)})\right]
%,\quad l\in\N. 
%\end{align}
%}
{%[How about]
%\begin{align}
%=\sum_{k\in\N} E^{(k)}(x)
%\end{align}
where %for non-negative measurable function $g$, we define 
\begin{align}
E^{(1)} (x) = \bP_x \left( X_{T(1)} <b \right)
,\quad E^{(l+1)}(x) = \bE_x \left[1_{\{X_{T(1)} >b\}} E^{(l)}(X_{T(1)})\right]
,\quad l\in\N. 
\end{align}
}

%where the last equaity holds by the strong Markov property. 
%we have 
%\begin{align}
%\sum_{k\in\bN}
%\bP_x& \left(X_{T(1)}> b, X_{T(2)}> b, \dots, 
%X_{T(k-1)}> b, X_{T(k)}<b \right)\\
%&=\sum_{k\in\bN}\bE_x \left[1_{\{X_{T(1)}> b\}}\bE_{X_{T(1)}}\left[ 1_{\{X_{T(1)}> b\}} \bE_{X_{T(1)}} \left[\dots 
%1_{\{X_{T(1)}> b\}} \bP_{X_{T(1)}}\left[ X_{T(1)} <b \right] \dots \right]\right]\right].\label{25}
%\end{align}
%Since the potential of $X$ has no atoms by Assumption \ref{Ass202}(1) and \cite[Proposition I.15]{Ber1996}, we have
%%\begin{align}
%%\bP_y (X_{T(1)}\geq b)-\bP_y (X_{T(1)}>b)=
%$\bP_y (X_{T(1)}= b)
%={\eta}\bE_y \left[ \int_0^\infty e^{-{\eta}t} 1_{\{X_t =b\}} \diff t\right]=0$ % \label{26}
%%\end{align}
% for $y\in\bR$.
%Hence,

{By Remark \ref{remark_mass},} 
 with $\{X_{T(1)}> b\}$ replaced by $\{X_{T(1)} \geq b\}$ {in the definition of $E^{(k)}$} and going backwards from \eqref{25}, we have
%{From \eqref{24}, \eqref{25} and \eqref{26}}, we have 
%\begin{align}
$\lim_{\varepsilon\downarrow 0}\bP_x \left(T_b= T_{b+\varepsilon} , T_b < \infty  \right)
%&= \sum_{k\in\bN}\bE_x \left[1_{\{X_{T(1)}\geq  b\}}\bE_{X_{T(1)}}\left[ 1_{\{X_{T(1)}\geq b\}} \bE_{X_{T(1)}} \left[\dots 
%1_{\{X_{T(1)}\geq  b\}} \bP_{X_{T(1)}}\left[ X_{T(1)} <b \right] \dots \right]\right]\right]\\
%&=\sum_{k\in\bN}\bP_x \left(X_{T(1)}\geq b, X_{T(2)}\geq b, \dots 
%X_{T(k-1)}\geq b, X_{T(k)}<b \right)\\
=\sum_{k\in\bN}\bP_x \left( T_b=T(k)\right)=\bP_x \left( T_b < \infty  \right)$. 

%\blue{[Maybe we can just say the following holds similarly?]}
%{[It's almost the same, but there are some differences. In the argument above, we had ``$\lim_{\varepsilon\downarrow 0} 1_{\{ X_{T(1)} \geq b+\varepsilon\}}=1_{\{ X_{T(1)} > b\}}$'' 
%in the integrals and used the fact about the potential to get \eqref{a001}. 
%However, in the following argument, we had $\lim_{\varepsilon\downarrow0}\bP_{X_{T(1)}}\left[ X_{T(1)} <b-\varepsilon \right]=\bP_{X_{T(1)}}\left[ X_{T(1)} <b \right]$ in the integrals and we did not use the argument about the potential. 
%The part using the monotone convergence theorem is the same. 
%]} \blue{I see. Lets keep these computations for now.}
On the other hand, by the monotone convergence theorem, %for $\varepsilon>0$, we have
\begin{align}
\lim_{\varepsilon\downarrow 0}\bP_x &\left(T_b= T_{b-\varepsilon}, T_b<\infty \right)
%= \lim_{\varepsilon\downarrow 0}\sum_{k\in\bN}\bP_x \left(T_b= T_{b-\varepsilon}=T(k)  \right)\\
\\
&=\sum_{k\in\bN} \lim_{\varepsilon\downarrow 0} \bP_x \left(X_{T(1)}\geq b, X_{T(2)}\geq b, \dots, 
X_{T(k-1)}\geq b, X_{T(k)}<b-\varepsilon \right)\\
%&= \sum_{k\in\bN}\bE_x \left[1_{\{X_{T(1)}\geq b\}}\bE_{X_{T(1)}}\left[ 1_{\{X_{T(1)}\geq b\}} \bE_{X_{T(1)}} \left[\dots 
%1_{\{X_{T(1)}\geq b\}} \bP_{X_{T(1)}}\left[ X_{T(1)} <b-\varepsilon \right] \dots \right]\right]\right]. 
&= \sum_{k\in\bN}\bP_x \left(X_{T(1)}\geq b, X_{T(2)}\geq b, \dots, 
X_{T(k-1)}\geq b, X_{T(k)}<b \right)  =\bP_x (T_b<\infty). 
\end{align}
%By the same argument as above with monotone convergence theorem, we have 
%\begin{align}
%&\lim_{\varepsilon\downarrow0}\bP_x \left(T_b =T_{b-\varepsilon}, T_b<\infty\right)\\
%&=\sum_{k\in\bN}\bP_x \left(X_{T(1)}\geq b, X_{T(2)}\geq b, \dots X_{T(k-1)}\geq b, X_{T(k)}<b \right)\\
%&=\sum_{k\in\bN}\bP_x \left( T_b=T(k)\right)= \bP_x \left(T_b< \infty \right).
%\end{align}
%The proof is complete. 
%Since $\lim_{\varepsilon\to0}\bP_x \left(T_b= T_{b+\varepsilon}, T_b<\infty \right)=\bP_x (T_b<\infty)$ for $x\in\bR$ by the above discussion and since the map $b\mapsto T_b$ is nonincreasing, the proof is complete. 
{Finally, because the map $b\mapsto T_b$ is nonincreasing, the proof is complete by monotone convergence. }
%\blue{[later the proof can be moved to the appendix.]}
%{[I think we should move it too.]}
%\end{proof}

\subsection{Proof of Lemma \ref{Lem401}} \label{proof_lem401}
%\blue{[this proof can be moved to the appendix.]}
%{[I agree with you.]}
Since $f^\prime_+$ is nondecreasing, % and since we have, for $b\in\bR$,
%\begin{align}
%\rho(b)=\bE_0 \left[\int_0^\infty e^{-qt} f^\prime_+ (U^0_t+b) \diff t \right], 
%\end{align}
the function {$\rho(b) = \bE_0 \left[\int_0^\infty e^{-qt} f^\prime_+ (U^0_t+b) \diff t \right]$} is nondecreasing. 
By {Corollary} \ref{Lem303}, monotone convergence, and Assumption \ref{Ass201} (3), we have $\lim_{b \uparrow\infty} \rho(b)= f^\prime_+(\infty) /q > -C$ and $\lim_{b \downarrow-\infty} \rho(b)= f^\prime_+(-\infty) /q< -C$. In what follows, we show the continuity of $\rho$. 

(i) We first prove that the potential of the process $U^b$ does not have mass. Recall the first control time  $T_b = T_b^{(1)}$.
% as in \eqref{first_control_time}{[We will delete \eqref{first_control_time}?]}. 
%We define inductively the subsequent control times by  $T^{[0]}_b=0$, $T^{[1]}_b=T_b$ and 
%We write 
%\begin{align}
%T^{[k]}_b :=\inf \{ t>T^{[k-1]}: R^b_t >R^b_{T^{[k-1]}}\}, \quad k\in\bN. 
%\end{align}
%{Note that $U_{T^{(k)}_b}^b = b$ on $\{T^{(k)}_b < \infty \}$ for $k \geq 1$.} \blue{[actually $T_b^{(k)}$ has already been defined. Just use this?]}
%{[That's right. I have erased the definition of $T_b^{[k]}$ and unified it with $T_b^{(k)}$. ]}
%where $T^{[0]}_b=0$. Note that $T^{[1]}_b=T_b$. 
For $x, y, b\in\bR$, we have 
\begin{align}
&\bE_x\left[\int_0^\infty e^{-qt} 1_{\{y\}} (U^b_t) \diff t\right]
=\sum_{k\in\bN}\bE_x \left[\int_{T^{(k-1)}_b}^{T^{(k)}_b}e^{-qt} 1_{\{y\}}(U^b_t) \diff t\right] \\
&\qquad\qquad=\bE_x \left[\int_0^{T_b}e^{-qt} 1_{\{y\}}(X_t) \diff t\right]
\\
&\qquad\qquad\qquad+\sum_{k\in\bN}\bE_x\left[e^{-q T_b}\right]{\left( \bE_b\left[e^{-qT_b}\right]\right)}^{k-1}\bE_b \left[\int_0^{T_b}e^{-qt} 1_{\{y\}}(X_t) \diff t\right],
\end{align}
which is equal to $0$ {by Remark \ref{remark_mass}.} 
%since the potential of $X$ does not have mass by Assumption \ref{Ass202}(1) and \cite[Proposition I.15]{Ber1996}. 
%\blue{[We assume $X$ is not a compound Poisson process? Is this all we need?]} 
%{[We assumed in Assumption \ref{Ass202}. I think we can prove the main theorem without the driftless condition. However, the continuity of $b\mapsto T_b$ may be vanished without this condition.]} \blue{[ok. Lets just add a reference containing the result that it does not have a mass when it is not a compound Poisson process.]}
%{[I added.]}
\par
(ii) %We prove that the function $\rho$ is continuous. 
Since $f^\prime_+$ is right-continuous and by the dominated convergence theorem {with Corollary \ref{Lem303}}, we have 
\begin{align}
\rho(b+\varepsilon)-\rho (b)
=\bE_0 \left[\int_0^\infty e^{-qt} (f^\prime_+ (U^0_t+b+\varepsilon)-f^\prime_+(U^0_t+b))  \diff t \right]
\xrightarrow{\varepsilon \downarrow 0} 0.
\end{align}
%as $\varepsilon \downarrow 0$. 
Let $D$ %\blue{[just change this to $D$?]} 
be the set of discontinuous point of $f^\prime_+$ {on $\R$}, which is 
%Then $D_{f^\prime_+}$ is 
{at most} countable set since $f^\prime_+$ is nondecreasing. 
By {Corollary} \ref{Lem303} and the dominated 
%\blue{monotone?}{[Since $f^\prime_+ (y) 1_{\{y\}}(U^b_t) $ is non-negative and decreasing w.r.t. $\varepsilon$. Thus, I guess we need to dominate using Lemma \ref{Lem303}. So, I think both is fine.]
convergence theorem, we have 
\begin{align*}
\rho(b)-\rho (b-\varepsilon)
&={\bE_0 \left[\int_0^\infty e^{-qt} (f^\prime_+ (U^0_t+b)-f^\prime_+(U^0_t+b-\varepsilon))  \diff t \right]} \\
&\xrightarrow{\varepsilon \downarrow 0} {\bE_0 \left[\int_0^\infty e^{-qt} \sum_{y\in D  }
(f^\prime_+ (y)-f^\prime_- (y)) 1_{\{y\}}(U^0_t+b)  \diff t \right]} \\ &{= \bE_b \left[\int_0^\infty e^{-qt} \sum_{y\in D  }
(f^\prime_+ (y)-f^\prime_- (y)) 1_{\{y\}}(U^b_t)  \diff t \right], }
\end{align*}
which is equal to $0$ since the potential of $U^b$ does not have mass as in (i). 

%\subsection{Proof of Lemma \ref{lemma_poly_growth}} \label{proof_lemma_poly_growth}
%\blue{[added]}

%{[For using the above identity, we need to prove 
%\begin{align}
% \bE_x \left[\int_0^\infty e^{-qt}|f^\prime_+ (X_t)|\diff t \right]. 
% \end{align}
% However, to prove this, we would need to do the same as in the proof of \cite[Lemma 3.3]{NobYam2020}. 
%]}
%
%{The proof of Lemma \ref{Lem303} is a direct corollary of \cite[Lemma 3.3]{NobYam2020}; it is given in Appendix \ref{Sec00A}. }
%{[I think we can delete Appendix \ref{Sec00A}. How about the following?] 
%We can prove Lemma \ref{Lem303} using Lemma \ref{Lem301}. We omit the proof since the proof is the same as that of \cite[Lemma 3.3]{NobYam2020}.} \blue{[I think we can delete the proof, but a slight modification of  \cite[Lemma 3.3]{NobYam2020} is necessary? Essentially, what we just need to say is that our $U^b$ is bounded from below by $X$ and bounded from above the classical reflected process $U^{b,\infty}$. We can just add a line.]}
%%\blue{[probably not exactly the same? Just say similar?]}
%%{[I added Appendix \ref{Sec00A}.]} \blue{what is $\tilde{U}$ in the proof?}
%{[Using the convexity of $f$ and Lemma \ref{Lem301}, I guess that the sentence written in the proof of A is an exact proof. If we use the boundedness method with $X$ and the resulting process $U^b$ of classical barrier strategy, we need to prove $\bE_x \left[\int_0^\infty e^{-qt}f^\prime_+ (X_t)\diff t \right]<\infty$, and for that we need to give the same proof for $X$ as for \cite[Lemma 3.3]{NobYam2020}.]}

\subsection{Proof of Lemma \ref{Lem502}} \label{proof_Lem502}
{
The proof is essentially the same as that of \cite[Lemma 9]{NobYam2020} by simply replacing the classical reflected process $U^{b^*,\infty}$ with the Poissonian version $U^{b^*}$. Following the same arguments, we obtain 
%\begin{align}
$v_{b^\ast}^{\prime\prime}(x)=\bE_x \left[ \int_0^{T_{b^\ast}} e^{-qt} f^{\prime\prime}(U^{b^\ast}_t) \diff t \right]$, 
%=\bE_0 \left[ \int_0^{T_{b^\ast-x}^-} e^{-qt} f^{\prime\prime}(U^{b^\ast-x}_t) \diff t \right],
%\end{align}
which can be shown to be continuous by the dominated convergence theorem using the assumption that $f''$ is of polynomial growth, \eqref{14},  and Lemma \ref{Lem302}. For more details, see \cite[Section A.6]{NobYam2020}.
%the function $v_{b^\ast}^{\prime\prime}(x)$ is continuous. 
}

\section*{Acknowledgements}
K. Noba stayed at Centro de Investigaci\'on en Matem\'aticas in Mexico as a JSPS Overseas Research Fellow and received support regarding the research
environment there. K. Noba is grateful for their support during his visit. 
K. Noba was supported by JSPS KAKENHI grant no.\ JP21K13807. 
K. Yamazaki was supported by JSPS KAKENHI grant no.\ JP20K03758, JP24K06844 and JP24H00328
and
 the start-up grant by the School of Mathematics and Physics of the University of Queensland. Both authors were supported by JSPS Open Partnership Joint Research Projects grant no. JPJSBP120209921 and JPJSBP120249936.

\end{document}